\documentclass[aihp]{imsart}

\RequirePackage{amsthm,amsmath,amsfonts,amssymb}
\RequirePackage[numbers,sort&compress]{natbib}
\RequirePackage[colorlinks,citecolor=blue,urlcolor=blue]{hyperref}
\RequirePackage{graphicx}
\usepackage{indentfirst}
\usepackage{caption}
\usepackage{mathrsfs}
\startlocaldefs
\theoremstyle{plain}

\newtheorem{theorem}{Theorem}[section]
\newtheorem*{theorem 2.12}{Theorem 2.12}
\newtheorem{lemma}[theorem]{Lemma}
\newtheorem{definition}[theorem]{Definition}
\newtheorem{example}[theorem]{Example}
\newtheorem{corollary}[theorem]{Corollary}
\theoremstyle{remark}
\newtheorem{remark}{Remark}

\endlocaldefs

\begin{document}
\begin{frontmatter}
\title{Most Probable KAM Tori in Stochastic Hamiltonian Systems}
\runtitle{Most Probable KAM Tori in Stochastic Hamiltonian Systems}

\begin{aug}
	\author[A]{\inits{F.}\fnms{Xinze}~\snm{Zhang}\ead[label=e1]{zhangxz24@mails.jlu.edu.cn}},
	\author[A,B]{\inits{S.}\fnms{Yong}~\snm{Li}\ead[label=e2]{liyong@jlu.edu.cn}}\thanks{corresponding author.}.
	\address[A]{School of Mathematics, Jilin University, ChangChun, People's Republic of China}
	
	\address[B]{Center for Mathematics and Interdisciplinary Sciences,\\    ~~~ Northeast Normal University, ChangChun, People's Republic of China\\
		\printead{e1,e2}}
	
\end{aug}

\begin{abstract}
This paper investigates in depth how stochastic perturbations affect the integrable structure of Hamiltonian systems and develops a KAM theory for stochastic Hamiltonian dynamics, in the sense of the most probable path. We first derive the Onsager-Machlup functional for stochastic Hamiltonian systems driven by time-dependent noise coefficients and identify the most probable path of the system trajectories. Building on this, we establish a large deviation principle and obtain an explicit rate function that quantitatively characterizes trajectory deviations, in particular for rare events. The main contribution of this work is to prove that, under stochastic noise, the original quasi-periodic invariant tori persist in the sense of the most probable path, thereby demonstrating the stability of KAM structures in random environments. Moreover, we show that the Onsager-Machlup functional coincides exactly with the large deviation rate function, thereby providing a quantitative characterization of both the structural persistence of quasi-periodic motions and the geometry of fluctuations in stochastic Hamiltonian systems. Overall, our results extend the classical KAM framework to stochastic settings and offer new insight into the behavior of complex dynamical systems under noise.
\end{abstract}



\end{frontmatter}

\section*{Statements and Declarations}
\begin{itemize}
	\item Ethical approval\\
	Not applicable.
	
	\item The Data Availability Statement.\\
	No datasets were generated or analysed during the current study.
	
	\item The Conflict of Interest Statement. \\
	We have no conflicts of interest to disclose.
	
	\item Funding\\
	The second author (Y. Li) was supported by National Natural Science Foundation of China (Grant No. 12071175, 12471183, 12531009).
\end{itemize}

\section{Introduction}

This paper investigates the persistence of invariant tori in nearly integrable Hamiltonian systems under stochastic perturbations. We refer to this class of models as nearly integrable stochastic Hamiltonian systems. More precisely, we work in action-angle coordinates $(I,\theta)\in \mathbb{D} \times \mathbb{T}^n$, where $\mathbb{D} \subset \mathbb{R}^n$ is the domain of the action variables and $\mathbb{T}^n$ denotes the $n$-dimensional torus.
\begin{equation}\label{1.1}
	\begin{cases}
		\mathrm{d} \theta(t) = \omega(I(t)) \, \mathrm{d}t + \frac{\partial P}{\partial I}(I(t), \theta(t)) \, \mathrm{d}t + \gamma \sigma_{\theta}(t) \, \mathrm{d} W_{\theta}(t), \\
		\mathrm{d} I(t) = - \frac{\partial P}{\partial \theta}(I(t), \theta(t)) \, \mathrm{d}t + \gamma \sigma_{I}(t) \, \mathrm{d} W_{I}(t),
	\end{cases}
\end{equation}
where the nearly integrable Hamiltonian is given by
\[
H(I,\theta)=H_0(I)+ P(I,\theta).
\]
Here, $I$ denotes the action variables and $\theta$ the angle variables. The function $H_0(I)$ is the completely integrable Hamiltonian, and $P(I,\theta)$ represents a small deterministic perturbation. Let
\[
\omega(I):=\nabla_I H_0(I)
\]
be the frequency map associated with the action variables. Moreover, $\gamma$ represents the intensity of the stochastic perturbation. The functions $\sigma_{\theta}(t)$ and $\sigma_{I}(t)$ are the diffusion coefficients, and $W_{\theta}(t)$ and $W_{I}(t)$ are independent Wiener processes.

Hamiltonian systems are foundational in classical mechanics, with applications across physics, astronomy, mechanical systems, and beyond. The core framework, dating back to the early 19th century, was introduced by William Rowan Hamilton. Early investigations focused primarily on completely integrable Hamiltonian systems, whose Hamiltonian depends only on the action variables, namely $H(I,\theta)=H_0(I)$. In this case, Hamilton's equations take the particularly simple form
\begin{equation}\label{1.2}
	\begin{cases}
		\mathrm{d}\theta(t)=\dfrac{\partial H_0}{\partial I}(I(t))\,\mathrm{d}t=\omega(I(t))\,\mathrm{d}t,\\
		\mathrm{d}I(t)=-\dfrac{\partial H_0}{\partial \theta}(I(t))\,\mathrm{d}t=0,
	\end{cases}
\end{equation}
so that the action variables remain constant along trajectories, while the angle variables evolve linearly in time,
\begin{equation*}
	\begin{cases}
		\theta(t)=\theta_0+\omega(I_0)t \ \ (\mathrm{mod}\ 2\pi),\\
		I(t)\equiv I_0.
	\end{cases}
\end{equation*}
Equivalently, the motion is a uniform linear flow on the invariant torus $\{I=I_0\}\cong \mathbb{T}^n$. Consequently, the phase space of a completely integrable system is foliated by a family of invariant tori, and the trajectory issued from any initial condition $(I_0,\theta_0)$ is confined to the corresponding torus $\mathbb{T}^n_{I_0}:=\{I=I_0\}$. Moreover, the arithmetic properties of the frequency vector $\omega(I_0)$ determine the qualitative recurrence and ergodic behavior of the orbit on $\mathbb{T}^n_{I_0}$. If the components of $\omega(I_0)$ satisfy a nontrivial rational relation, then the motion is periodic (or, more generally, quasi-periodic but confined to a lower-dimensional subtorus). If $\omega(I_0)$ is rationally independent, then the motion is quasi-periodic and the orbit is dense in $\mathbb{T}^n_{I_0}$.

However, real physical systems are rarely strictly integrable. In his study of the three-body problem and, more broadly, perturbations of Hamiltonian systems, Henri Poincar\'e raised and advanced a fundamental dynamical question: under a small perturbation, does the long-time behavior of the system retain some form of regularity? Such systems are described by a nearly integrable Hamiltonian of the form $H(I,\theta) = H_0(I) + P(I,\theta)$, and the associated Hamiltonian equations read
\begin{equation}\label{1.3}
	\begin{cases}
		\mathrm{d} \theta(t) = \frac{\partial H}{\partial I}(I(t))\, \mathrm{d}t = \omega(I(t)) \, \mathrm{d}t + \frac{\partial P}{\partial I}(I(t), \theta(t)) \, \mathrm{d}t, \\
		\mathrm{d} I(t) = -\frac{\partial H}{\partial \theta}(I(t)) \, \mathrm{d}t = - \frac{\partial P}{\partial \theta}(I(t), \theta(t)) \, \mathrm{d}t.
	\end{cases}
\end{equation}

To address these challenges, the KAM theory was a major breakthrough in the 20th century. Kolmogorov \cite{1} hypothesized that when Hamiltonian systems are subject to small perturbations, some of the invariant tori (regular orbital structures) would persist and avoid chaotic behavior. Arnold \cite{2} and Moser \cite{3} subsequently provided rigorous proofs of this conjecture, formalizing what is now known as the KAM theory. The core result demonstrates that, under small perturbations, most invariant tori persist, preserving quasi-periodic motions. This result has profoundly advanced the study of the stability of Hamiltonian systems under perturbations. For further developments and related results, see, for example, \cite{4,5,6,7,8,9,10,13}.

The original KAM framework was designed primarily for deterministic perturbations, leaving the question of its applicability in stochastic settings unresolved. As a result, validating the extension of the KAM theory under stochastic perturbations and quantifying the stability of invariant tori has emerged as a central problem in stochastic Hamiltonian systems research. Recent studies have achieved new progress in the study of stochastic Hamiltonian systems. Wu \cite{15} established a framework for both large and moderate deviations in stochastic Hamiltonian systems, providing a quantitative assessment of the probability of rare events. Talay \cite{16} explored how stochastic Hamiltonian systems asymptotically converge to an invariant measure, with a focus on the exponential nature of this convergence. Li \cite{17} proposed an averaging principle framework for completely integrable stochastic Hamiltonian systems, simplifying the analysis of their long-term behavior under small stochastic perturbations. 
Davini and Siconolfi \cite{17.5} performed a qualitative investigation of critical Hamilton-Jacobi equations with stationary ergodic Hamiltonians in one dimension.
Zhang \cite{18} examined stochastic flows within Hamiltonian systems and introduced new computational methods based on the Bismut formula. Some of the latest research references are \cite{18.5, 19, 20, 21, 22} and so on.

The value of this paper lies in a novel analytical framework for characterizing the stability of invariant tori in stochastic Hamiltonian systems, which may be viewed as a stochastic KAM theory. By coherently integrating the Onsager-Machlup functional, large deviation theory, and KAM theory, we provide new insights into the dynamical behavior of such systems under stochastic perturbations.

The exposition of this paper follows a progressive scheme from a general framework to structured special cases. We begin by deriving the Onsage-Machlup functional for a general class of stochastic Hamiltonian systems \eqref{3.1} (see Theorem \ref{T3.1}). This part is inspired by our previous work \cite{42}. In particular, we observe that, in the Hamiltonian setting, the Onsager-Machlup functional exhibits a distinguished structure that stems from the underlying symplectic geometry.

Next, we minimize the Onsager-Machlup functional to identify the most probable continuous path. We show that this path coincides with the solution of the deterministic Hamiltonian system, namely the dynamics obtained by removing the stochastic perturbations. It is important to emphasize that the noise intensity does not alter the most probable path itself; rather, it only changes the likelihood that the sample path remains close to this most probable path. This highlights a form of structural stability for Hamiltonian dynamics in random environments.

Furthermore, by combining the Onsager-Machlup functional with the Freidlin-Wentzell large deviation theory, we establish a large deviation principle for the stochastic Hamiltonian system. We find that the associated rate function agrees exactly with the Onsager-Machlup functional (see Theorem \ref{T5.1}).

Building on these results, in order to incorporate the KAM mechanism and investigate the persistence of invariant tori, we further restrict our attention to the near-integrable setting. That is, for the nearly integrable stochastic Hamiltonian system \eqref{1.1}, we establish the following stochastic version of the KAM theorem.

\begin{theorem}\label{T6.1}
	Consider the stochastic Hamiltonian system given by $\eqref{1.1}$, where the diffusion coefficients $\sigma_{\theta}(t)$ and $\sigma_{I}(t)$ satisfy condition (C2), and the Hamiltonian $H(I, \theta) = H_0(I) + P(I, \theta)$ is sufficiently smooth and satisfies condition (C1). Then the Onsager-Machlup functional for system $\eqref{1.1}$ is given by
	\begin{displaymath}
		\begin{aligned}
			\int_{0}^{1} OM(\varphi_{\theta}, \varphi_{I}) \,{\rm d}t &= \frac{1}{\gamma^2} \left( \int_0^1 \left| \sigma_{\theta}^{-1}(t) \left( \dot{\varphi_{\theta}} - \omega(\varphi_{I}) - \frac{\partial P}{\partial \varphi_{I}}(\varphi_{\theta}, \varphi_{I}) \right) \right|^2 \,{\rm d}t\right.\\
			&\qquad \left. + \int_0^1 \left| \sigma_{I}^{-1}(t) \left( \dot{\varphi_{I}} + \frac{\partial P}{\partial \varphi_{\theta}}(\varphi_{\theta}, \varphi_{I}) \right) \right|^2 \,{\rm d}t\right).
		\end{aligned}
	\end{displaymath}

	Furthermore, by applying the variational principle to minimize the Onsager-Machlup functional, the most probable transition path can be obtained, corresponding to the solution of the nearly integrable Hamiltonian system $\eqref{1.3}$.
	
	Additionally, when conditions (H1)-(H5) hold, the invariant tori of the original integrable system $\eqref{1.2}$ whose frequency vectors $\omega$ are $(\alpha,\tau)$-Diophantine remain preserved under both deterministic and stochastic perturbations, albeit with slight deformation, in the sense of most probable.
	
	Finally, let $X^{\gamma}(\cdot)$ denote the solution to system~\eqref{1.1}. As $\gamma \to 0$, the family of probability measures on the path space induced by $X^{\gamma}(\cdot)$ satisfies a large deviation principle with speed $\gamma^{2}$ and rate function
	\begin{displaymath}
		\begin{aligned}
			J(\varphi) &= \frac{1}{2}\left( \int_0^1 \left| \sigma_{\theta}^{-1}(t) \left( \dot{\varphi_{\theta}} - \omega(\varphi_{I}) - \frac{\partial P}{\partial \varphi_{I}}(\varphi_{\theta}, \varphi_{I}) \right) \right|^2 \,{\rm d}t\right.\\
			&\qquad \left. + \int_0^1 \left| \sigma_{I}^{-1}(t) \left( \dot{\varphi_{I}} + \frac{\partial P}{\partial \varphi_{\theta}}(\varphi_{\theta}, \varphi_{I}) \right) \right|^2 \,{\rm d}t\right),
		\end{aligned}
	\end{displaymath}
	where $\sigma_{\theta}^{-1}(t)$ and $\sigma_{I}^{-1}(t)$ denote the inverses of $\sigma_{\theta}(t)$ and $\sigma_{I}(t)$, respectively. Conditions (C1), (C2) and Conditions (H1)-(H5) are stated in Section 3 and Section 6, respectively.
\end{theorem}
\begin{remark}
	Here a vector $\omega \in \mathbb{R}^d$ is said to be $(\alpha,\tau)$-Diophantine if there exist constants
	$\alpha>0$ and $\tau>0$ such that
	\[
	|\omega \cdot k| \ge \frac{\alpha}{|k|_1^{\tau}},\quad |k|_1:= \sum_{j=1}^n |k_j| ,\quad \forall\, k \in \mathbb{Z}^d \setminus \{0\}.
	\]
\end{remark}
\begin{remark}
	In the original integrable system $\eqref{1.2}$, the invariant tori with frequency vectors satisfying the $(\alpha,\tau)$-Diophantine condition persist, in the sense of most probable, under both deterministic and stochastic perturbations, albeit with slight deformation. More specifically, the solution of the stochastic Hamiltonian system is a stochastic process corresponding to infinitely many possible sample paths. Although each individual sample path may deviate from the torus, the most probable sample path remains on a slightly deformed torus. This conclusion acknowledges that noise may, in principle, disrupt the deterministic torus structure, while also revealing that the dominant dynamical behavior continues, in a probabilistic sense, to follow the original geometric configuration.
\end{remark}

The structure of this paper is as follows: In Section 2, we review the basic definitions of the relevant function spaces and norms, introduce the concepts of the Onsager-Machlup functional and large deviation theory, present a general deterministic version of the KAM theorem, and provide several key technical lemmas. In Section 3, we derive the Onsager-Machlup functional for stochastic Hamiltonian systems. In Section 4, we prove that the most probable path of a stochastic Hamiltonian system corresponds to the stable solution of its associated deterministic Hamiltonian system. A specific example of a one-dimensional stochastic harmonic oscillator is provided to illustrate our results. In Section 5, we derive the large deviation principle for stochastic Hamiltonian systems. Finally, in Section 6, we extend the preceding results to the case of nearly integrable stochastic Hamiltonian systems, thereby completing the proof of Theorem \ref{T6.1} and illustrating our findings with an explicit example.
\section{Preliminaries}

\subsection{Approximate limits in Wiener space}
In this section, we recall some fundamental definitions and results concerning approximate limits in Wiener space. Specifically, we focus on the measurable semi-norm, which pertains to the exponentials of random variables in the first and second Wiener chaos (reference \cite{16}).

Let $W = \left\{ W_t, ~t \in [0, 1] \right\}$ be a Brownian motion (Wiener process) defined in the complete filtered probability space $(\Omega, \mathcal{F}, \left\{ \mathcal{F}_t \right\}_{t \geq 0}, \mathbb{P})$. Here, $\Omega$ represents the space of continuous functions vanishing at zero, and $\mathbb{P}$ denotes the Wiener measure. Let $\mathbb{H} := L^2([0,1], \mathbb{R}^{n} \times \mathbb{R}^{n})$ be a Hilbert space and $\mathbb{H}^1_0$ be the Cameron-Martin space defined as
\begin{displaymath}
	\begin{aligned}
		\mathbb{H}^1_0 &:= \left\{ f\in AC([0,1];\mathbb{R}^{n} \times \mathbb{R}^{n}) : f(0)=0,\ \dot f\in \mathbb{H} \right\}.
	\end{aligned}
\end{displaymath}
The scalar product in $\mathbb{H}^1_0$ is defined as
\begin{displaymath}
	\langle f, g \rangle_{\mathbb{H}^1_0} = \langle \dot{f}, \dot{g} \rangle_{\mathbb{H}}, \qquad
	\|f\|_{\mathbb{H}_0^1}^2=\int_0^1 |\dot{f}(t)|^2\,\mathrm{d}t.
\end{displaymath}
for all $f, g \in \mathbb{H}^1_0$.
Let $\mathcal{P}:\mathbb{H}^1_0 \to \mathbb{H}^1_0$ be an orthogonal projection with $dim \mathcal{P}\mathbb{H}^1_0 < \infty$ and the specific expression
\begin{displaymath}
	\mathcal{P}h = \sum_{i = 1}^{n} \langle h_i, f \rangle h_i,
\end{displaymath}
where $(h_1, ..., h_n)$ is a set of orthonormal basis in $\mathcal{P}\mathbb{H}^1_0$. In addition, we can also define the
$\mathbb{H}^1_0$-valued random variable
\begin{displaymath}
	\mathcal{P}^W = \sum_{i = 1}^{n} \bigg( \int_{0}^{1} {h_i^{\prime}} \,{\rm d}W_s \bigg) h_i,
\end{displaymath}
where $\mathcal{P}^W$ does not depend on $(h_1, ..., h_n)$.
\begin{definition}\label{definition 2.1}
	We say that a sequence of orthogonal projections $\mathcal{P}_n$ on $\mathbb{H}^1_0$ is an approximating sequence of projections, if $dim \mathcal{P}_n \mathbb{H}^1_0 < \infty$ and $\mathcal{P}_n$ converges strongly to the identity operator $\mathrm{Id}$ in $\mathbb{H}^1_0$ as $n \to \infty$.
\end{definition}

\begin{definition}\label{definition 2.2}
	We say that a semi-norm $\mathcal{N}$ on $\mathbb{H}^1_0$ is measurable, if there exists a random variable $\tilde{\mathcal{N}}$, satisfying $\tilde{\mathcal{N}} < \infty $ a.s, such that for any approximating sequence of projections $\mathcal{P}_n$ on $\mathbb{H}^1_0$, the sequence $\mathcal{N}(\mathcal{P}^W_n)$ converges to $\tilde{\mathcal{N}}$ in probability and $\mathbb{P}(\tilde{\mathcal{N}} \leq \epsilon) > 0$ for any $\epsilon > 0$. Moreover, if $\mathcal{N}$ is a norm on $\mathbb{H}^1_0$, then we call it a measurable norm.
\end{definition}

For proving the measurability of the semi-norm defined in this paper, it is necessary to introduce the following lemma (see \cite{38}).
\begin{lemma}\label{lemma 2.2}
	Let $\mathcal{N}_n$ be a nondecreasing sequence of measurable semi-norms. Suppose that $\tilde{\mathcal{N}} := \mathbb{P}\text{-}\!\lim\limits_{n \to \infty} \tilde{\mathcal{N}}_n$ exists and $\mathbb{P}(\tilde{\mathcal{N}} \leq \epsilon) > 0$ for any $\epsilon > 0$. In addition, if the limit $\lim\limits_{n \to \infty} \mathcal{N}_n$ exists on $\mathbb{H}^1_0$, then $\mathcal{N} := \lim\limits_{n \to \infty} \mathcal{N}_n$ is a measurable semi-norm.
\end{lemma}
\begin{definition}\label{definition 2.4}
	Let $f$ be a function defined on $\Omega$. For $0 < \alpha < 1$, we introduce Hölder norm ($\alpha$-Hölder)
	\begin{displaymath}
		\Vert f \Vert_{\alpha} = \Vert f \Vert_{\alpha; \Omega} = \Vert f \Vert_{0; \Omega} + \left[ f \right]_{\alpha; \Omega},
	\end{displaymath}
	where $\Vert f \Vert_{0; \Omega}$ represents the supremum norm of $f$ on $\Omega$, and $\left [f \right]_{\alpha; \Omega}$ represents the Hölder semi-norm of $f$ on $\Omega$. The specific expression is as follows:
	\begin{displaymath}
		\Vert f \Vert_{0; \Omega}  = \sup\limits_{x \in \Omega}\vert f(x) \vert, \quad
		\left[ f \right]_{\alpha; \Omega} = \sup\limits_{x, y \in \Omega, x \neq y} \frac{\left| f(x) - f(y) \right|}{\vert x - y \vert^{\alpha}}.
	\end{displaymath}
\end{definition}

\begin{definition}\label{def:Calpha_path_space}
	Fix $\alpha\in(0,1)$. We define the H\"older path space
	\[
	C^\alpha\big([0,1];\mathbb{R}^n\times\mathbb{R}^n\big)
	:=\Big\{\varphi:[0,1]\to\mathbb{R}^n\times\mathbb{R}^n \ \text{continuous}:\ 
	\|\varphi\|_{\alpha;[0,1]}<\infty\Big\},
	\]
	where $\|\cdot\|_{\alpha;[0,1]}$ is the $\alpha$-H\"older norm introduced in Definition~\ref{definition 2.4}
	(with $\Omega=[0,1]$), computed using the Euclidean norm on $\mathbb{R}^n\times\mathbb{R}^n$.
\end{definition}

\begin{definition}\label{def:operator_norm}
	Let \( A \in \mathbb{R}^{n \times n} \) be a real \( n \times n \) matrix. The \emph{operator norm of \( A \) induced by the Euclidean vector norm} (also called the \emph{spectral norm}) is defined by
	\begin{displaymath}
		\left| \left\|  A \right\| \right| := \sup_{x \in \mathbb{R}^n \setminus \{0\}} \frac{\| A x \|_2}{\| x \|_2}
		= \sup_{\| x \|_2 = 1} \| A x \|_2,
	\end{displaymath}
	where \( \| x \|_2 = \left( \sum_{i=1}^n x_i^2 \right)^{1/2} \) denotes the standard Euclidean norm of \( x \in \mathbb{R}^n \).
	
	If \( A \) is symmetric, then \( \left| \left\|  A \right\| \right| \) equals the largest absolute value of the eigenvalues of \( A \). In particular, if \( A \) is symmetric positive definite, then
	\begin{displaymath}
		\left| \left\| A^{-1} \right\| \right| = \frac{1}{\lambda_{\min}(A)},
	\end{displaymath}
	where \( \lambda_{\min}(A) > 0 \) denotes the smallest eigenvalue of \( A \).
\end{definition}

Throughout this paper, unless stated otherwise, $\left|~ \cdot ~\right| $ denotes the Euclidean norm, $\left\|~ \cdot ~\right\|_{\alpha}$ denotes the Hölder norm, and $\left| \left\|~ \cdot~\right\| \right| $ denotes the matrix operator norm.

\subsection{Onsager-Machlup Functional}
In the problem of finding the most probable path of a diffusion process, the probability of a single path is zero. Instead, we can search for the probability that the path lies within a certain region, which could be a tube along a differentiable function. This tube is defined as
\[
\mathbb{K}(\varphi, \epsilon) = \{ x - x_0 \in \mathbb{H}^1 \mid \varphi - x_0 \in \mathbb{H}^1_0, \|x - \varphi\| \leq \epsilon, \epsilon > 0 \},
\]
where \( \Vert \cdot \Vert \) is an appropriate norm. Once \( \epsilon > 0 \) is given, the probability of the tube can be expressed as
\[
\mu_x(\mathbb{K}(\varphi, \epsilon)) = \mathbb{P}\left( \{ \omega \in \Omega \mid X_t(\omega) \in \mathbb{K}(\varphi, \epsilon) \} \right),
\]
allowing us to compare the probabilities of the tubes for all \( \varphi- x_0 \in \mathbb{H}^1_0 \), since \( \mathbb{K}(\varphi, \epsilon) \in \mathcal{B} \). Here, \( \mu_x \) denotes the probability measure under the initial state \( x \), and \( \mathcal{B} \) is the Borel $\sigma$-field defined on the function space \( \mathbb{H}^1_0 \), containing all measurable sets generated by open sets.

Thus, the Onsager-Machlup function can be defined as the Lagrangian function that provides the most probable tube. This function plays a crucial role in analyzing path probabilities and rare events in stochastically perturbed dynamical systems. Similar to the action functional in classical mechanics, it quantifies the likelihood of various paths within the probabilistic framework.

Onsager and Machlup \cite{23,24} first introduced this tool in 1953 to describe the probability density of diffusion processes with linear drift and constant diffusion coefficients. In 1957, Tisza and Manning \cite{25} extended its application to nonlinear equations, while Stratonovich \cite{26} provided a rigorous mathematical framework for the theory in the same year. In recent years, with deeper research in this area, the Onsager-Machlup functional has been increasingly applied in stochastic systems, particularly for analyzing the most probable path under stochastic perturbations, as discussed in related studies \cite{43,44,45,46,42}.

Based on the above background, we now formally define the Onsager-Machlup function and functional:

\begin{definition}
	Consider a tube surrounding a reference path \( \varphi_t \) with initial value \( \varphi_0 = x \), where \( \varphi_t - x \in \mathbb{H}^1_0 \). Assuming \( \epsilon \) is small enough, we estimate the probability that the solution process \( X_t \) falls within this tube as
	\begin{displaymath}
		\mathbb{P} \left(  \Vert X - \varphi\Vert \leq \epsilon\right)  \propto C(\epsilon) {\rm exp} \left\{ -\frac{1}{2} \int_{0}^{1} {OM(t, \varphi, \dot{\varphi})} \,{\rm d}t \right\},
	\end{displaymath}
	where \( \propto \) denotes equivalence for small enough \( \epsilon \), and \( \Vert \cdot \Vert \) is an appropriate norm. Here, the integrand \( OM(t, \varphi, \dot{\varphi}) \) is called the Onsager-Machlup function, while the integral \( \int_{0}^{1} {OM(t, \varphi, \dot{\varphi})} \,{\rm d}t \) is called the Onsager-Machlup functional. In the framework of classical mechanics, we also refer to these as the Lagrangian function and the action functional, respectively.
\end{definition}

\subsection{Large deviation principle}
The origins of large deviation theory and its associated research can be traced back to the early 20th century. Cram\'er \cite{27} and Sanov \cite{28} made foundational contributions to the study of large deviations in sequences of independent and identically distributed random variables. Later, Donsker and Varadhan \cite{29,30,31} systematically investigated large deviations in the context of Markov processes and explored their relationship with ergodic theory. Their work introduced essential concepts such as Varadhan's integral lemma and the contraction principle, which are not only central results in large deviation theory but also establish profound connections with other areas of mathematics (see \cite{32, 33, 34, 35}). In the 1970s, Freidlin and Wentzell \cite{36} extended this theory to stochastic dynamical systems and stochastic differential equations, particularly in the setting of small perturbations. The Freidlin-Wentzell framework describes the probability of a system deviating from its most probable path and introduces the rate function to quantify the distribution of deviations from typical behavior. Let \(\mathbb{M}\) be a Polish space. Below, we provide the precise definitions of the rate function and the large deviation principle.
\begin{definition}
	A function \( J : \mathbb{M} \rightarrow [0, +\infty) \) is called a rate function if \( J \) is lower semicontinuous. Moreover, a rate function \( J \) is called a good rate function if the level set \( \{ x \in \mathbb{M} : J(x) \leq K \} \) is compact for each constant \( K < \infty \).
\end{definition}

\begin{definition}\label{D2.8}
	The random variable sequence \( \{ X^{\gamma} \} \) is said to satisfy the LDP on \( \mathbb{M} \) with rate function \( J \) if the following lower and upper bound conditions hold:
	\begin{itemize}
		\item[(i)] (Lower bound) For any open set \( \mathbb{G} \subset \mathbb{M} \),
		\[
		\liminf_{\gamma \rightarrow 0} \gamma^2 \log \mathbb{P}(X^{\gamma} \in \mathbb{G}) \geq - \inf_{x \in \mathbb{G}} J(x).
		\]
		\item[(ii)] (Upper bound) For any closed set \( \mathbb{F} \subset \mathbb{M} \),
		\[
		\limsup_{\gamma \rightarrow 0} \gamma^2 \log \mathbb{P}(X^{\gamma} \in \mathbb{F}) \leq - \inf_{x \in \mathbb{F}} J(x).
		\]
	\end{itemize}
\end{definition}

\subsection{KAM Theory}

In Hamiltonian mechanics, an invariant torus is an invariant submanifold of phase space on which the trajectories are quasi-periodic. Such tori may be viewed as higher-dimensional analogues of periodic orbits: when the components of the frequency vector are rationally independent, the motion on the torus is quasi-periodic. KAM theory is concerned with the stability of these invariant tori under small perturbations. For a nearly integrable Hamiltonian system, whose Hamiltonian consists of an integrable part plus a small perturbation, KAM theory asserts that if the perturbation is sufficiently small and appropriate hypotheses hold, then a family of invariant tori of sufficiently large measure persists, up to a small deformation.

We cite the following theorem from \cite{6}:

\begin{theorem}\label{T2.6}
	Consider a Hamiltonian of the form $ H(I, \theta) = H_0(I) + P(I, \theta) $, where $ I \in \mathbb{D} \subset \mathbb{R}^n $ are the action variables and $ \theta \in \mathbb{T}^n $ are the angle variables. Here, $ H_0(I) $ and $ P(I, \theta) $ are $ C^l $-smooth functions with $ H_0, P \in C^l(\mathbb{D} \times \mathbb{T}^n) $, where $ \mathbb{D} $ is a non-empty bounded domain in $ \mathbb{R}^n $. If $H_0$ is non-degenerate and $l > 2\nu > 2n$, then all the KAM tori of the integrable system $H_0$ whose frequency are $(\alpha, \tau)$-Diophantine, with $\alpha \simeq \eta^{1/2 - \nu / l}$ and $\tau := \nu - 1$, do survive, being only slightly deformed, where $\eta$ is the $C^l$-norm of the perturbation $P$. Moreover, letting $\mathcal{K}$ be the corresponding family of KAM tori of $H$, we have 
	\[
	\mathrm{meas}(\mathcal{D} \times \mathbb{T}^n \setminus \mathcal{K}) = O(\eta^{1/2 - \nu / l}).
	\]
\end{theorem}

This theorem provides a more refined theoretical foundation for the persistence of invariant tori in finitely differentiable Hamiltonian systems, extending the classical KAM theory to the case where the Hamiltonian is only finitely smooth. It demonstrates that, even under conditions of finite differentiability, a significant portion of the invariant tori remains stable. This stability implies that, despite perturbations, many quasi-periodic motions can still exist and maintain their regularity in phase space. This result enhances the robustness of the KAM theory, showing that the structure of Hamiltonian systems can exhibit notable stability even under less stringent smoothness conditions.

\subsection{Technical Lemmas}
In this section, we will introduce several commonly utilized technical lemmas. Throughout this paper, if not mentioned otherwise, $\mathbb{E} \left(A ~\big|~ B\right)$ represents the conditional expectation of $A$ given $B$. $C$ represents a positive constant and varies with these different rows.

When we derive the Onsager-Machup functional of SDEs, the following lemma is the most basic one, as it ensures that we handle each term separately. Its proof can be found in \cite{39}.
\begin{lemma}\label{lemma 2.3}
	For a fixed integer $N \geq 1$, let $X_1, ..., X_N \in \mathbb{R}$ be $N$ random variables defined on $(\Omega, \mathcal{F}, \left\{ \mathcal{F}_t \right\}_{t \geq 0}, \mathbb{P})$ and $\left\{D_{\epsilon}; \epsilon > 0 \right\}$ be a family of sets in $\mathcal{F}$. Suppose that for any $c \in \mathbb{R}$ and any $i = 1, ..., N$,
	\begin{displaymath}
		\limsup\limits_{\epsilon \to 0} \mathbb{E}\left({\rm exp}\left\{ c X_i \right\}~\big|~D_{\epsilon} \right) \leq 1.
	\end{displaymath}
	Then
	\begin{displaymath}
		\limsup\limits_{\epsilon \to 0} \mathbb{E}\left({\rm exp}\left\{ \sum_{i = 1}^{N}c X_i \right\} ~\big|~ D_{\epsilon} \right)= 1.
	\end{displaymath}
\end{lemma}
The following two lemmas are fundamental parts of calculating Onsager-Machup functional. Their proofs can be found in \cite{37}.
\begin{lemma}\label{lemma 2.4}
	Let $\mathcal{N}$ be a measurable norm on $\mathbb{H}_0^1$. For any $f \in L^2([0, 1])$, we have
	\begin{displaymath}
		\lim\limits_{\epsilon \to 0}\mathbb{E} \left( {\rm exp} \left\{ {\int_{0}^{1} {f(s)} \,{\rm d}W_s } \right\} ~\big|~ \tilde{\mathcal{N}} < \epsilon \right) = 1,
	\end{displaymath}
	where $\tilde{\mathcal{N}}$ is defined by Definition \ref{definition 2.2}.
\end{lemma}
\begin{definition}
	We say that an operator $\mathcal{S} : \mathbb{H} \to \mathbb{H}$ is nuclear, if
	\begin{displaymath}
		\sum_{n = 1}^{\infty} \left|\langle  \mathcal{S}e_n, k_n  \rangle \right| < \infty,
	\end{displaymath}
	for any orthonormal sequences $B_1 = \left\lbrace e_n \right\rbrace_{n \in \mathbb{N}}$ and $B_2 = \left\lbrace k_n \right\rbrace_{n \in \mathbb{N}}$ in $\mathbb{H}$.
\end{definition}
We define the trace of a nuclear operator $\mathcal{S}$ as
\begin{displaymath}
	\mathrm{Tr}~\mathcal{S} = \sum_{n = 1}^{\infty} \langle  \mathcal{S}e_n, e_n  \rangle
\end{displaymath}
for any orthonormal sequence $B = \left\lbrace e_n \right\rbrace_{n \in \mathbb{N}}$ in $\mathbb{H}$. The definition of trace is independent of the orthonormal sequences we choose. For a given symmetric function $f \in L^2([0, 1]^2)$, the Hilbert-Schmidt operator $\mathcal{S}(f) : \mathbb{H} \to \mathbb{H}$ defined by
\begin{displaymath}
	\left(\mathcal{S} \left(f \right) \right)(h)(t) = \int_{0}^{t} {f(t, s)h(s)} \,{\rm d}s
\end{displaymath}
is nuclear if $\sum_{n = 1}^{\infty} \langle  \mathcal{S}e_n, e_n  \rangle < \infty$ for any orthonormal sequence $B = (e_n)_n$ in $\mathbb{H}$. When the function $f$ is continuous and the operator $\mathcal{S}(f)$ is nuclear, its trace is given by (see \cite{40})
\begin{displaymath}
	\mathrm{Tr}~f := \mathrm{Tr}~\mathcal{S}(f) = \int_{0}^{1} {f(t, t)} \,{\rm d}t.
\end{displaymath}
Furthermore, when ${f(s, t)}$ is a continuous $n \times n$ covariance kernel in the square $0 \leq s, t \leq 1$, the corresponding operator $\mathcal{S}$ is nuclear and the expression for its trace is
\begin{displaymath}
	\mathrm{Tr}~f := \mathrm{Tr}~\mathcal{S}(f) = \int_{0}^{1} {\mathrm{Tr}~f(t, t)} \,{\rm d}t.
\end{displaymath}
\begin{lemma}\label{lemma 2.5}
	Let $f$ be a symmetric function in $L^2([0, 1]^2)$ and let $\mathcal{N}$ be a measurable norm. If $\mathcal{S}(f)$ is nuclear, then
	\begin{displaymath}
		\lim\limits_{\epsilon \to 0} \mathbb{E} \left(  {\rm exp}\left\{ \int_{0}^{1} {\int_{0}^{1} {f(s, t)} \,{\rm d}W_s} \,{\rm d}W_t   \right\} ~\big|~ \tilde{\mathcal{N}} < \epsilon \right)  = e^{-\mathrm{Tr}~f}.
	\end{displaymath}
\end{lemma}

The following two lemmas are about the probability estimation of Brownian motion balls. The proof of the lemmas can be found in \cite{41}.
\begin{lemma}\label{lemma 2.6}
	Let $\left\{W(t): t \leq 0\right\}$ be a sample continuous Brownian motion in $\mathbb{R}$ and set 
	\begin{displaymath}
		\Phi_{\alpha}(\epsilon) = {\rm log}\mathbb{P}\left(\Vert W \Vert_{\alpha} \leq \epsilon \right).
	\end{displaymath}
	If $0 < \alpha < \frac{1}{2}$, then 
	\begin{displaymath}
		\lim\limits_{\epsilon \to 0} \epsilon^{\frac{2}{1 - 2\alpha}} \Phi_{\alpha}(\epsilon) = -C_{\alpha}
	\end{displaymath}
	exists with
	\begin{displaymath}
		2^{-\frac{2\left(1 - \alpha\right)}{\left(1 - 2\alpha\right)}} \Lambda_{\alpha} \leq C_{\alpha} \leq \left(2^{-\frac{1}{2}} \left( 2^{\alpha} - 1\right) \left( 2^{1 - \alpha} - 1\right)\right)^{-\frac{2(1 - \alpha)}{(1 - 2\alpha)}} \Lambda_{\alpha},
	\end{displaymath}
	where
	\begin{displaymath}
		\Lambda_{\alpha} = \left(\frac{2}{\pi}\right)^{\frac{1}{2}} \int_{0}^{\infty} {\frac{x^{\frac{2}{1 - 2\alpha}} e^{-\frac{x^2}{2}}}{1 - G(x)}} \,{\rm d}x \qquad\text{and}\qquad
		G(x) = \left(\frac{2}{\pi}\right)^{\frac{1}{2}} \int_{x}^{\infty} {e^{-\frac{y^2}{2}}} \,{\rm d}y.
	\end{displaymath}
\end{lemma}

\begin{lemma}\label{lemma 2.7}
	Let \( \sigma(t) \in C([0, T], \mathbb{R}^{n \times n}) \) be a diagonal matrix, and assume that there exist positive constants \( M \) and \( m \) such that its diagonal elements satisfy \( m \leq \sigma_i(t) \leq M \), for \( 1 \leq i \leq n \). We have
	\begin{displaymath}
		\mathbb{P}\left( \Vert W \Vert_{\alpha} \leq \frac{\epsilon}{M} \right) \leq \mathbb{P}\left(\left\| \int_{0}^{t} \left\langle \sigma(s), \,{\rm d}W_s \right\rangle \right\|_{\alpha} \leq \epsilon \right) \leq  \mathbb{P}\left(\Vert W \Vert_{\alpha} \leq \frac{\epsilon}{m}\right)
	\end{displaymath}
	for any $0 \leq t \leq 1$. According to Lemma $\ref{lemma 2.6}$, we have
	\begin{displaymath}
		\lim\limits_{\epsilon \to 0} \mathbb{P}\left(\left\| \int_{0}^{t} \left\langle \sigma(s), \,{\rm d}W_s \right\rangle \right\|_{\alpha} \leq  \epsilon \right) \geq \lim\limits_{\epsilon \to 0} \mathbb{P}\left(\Vert W \Vert_{\alpha} \leq \frac{\epsilon}{M} \right)
		\geq e^{-c \left( \frac{\epsilon}{M} \right)^{-\frac{2}{1 - 2\alpha}}},
	\end{displaymath}
	where $c = \left(2^{-\frac{1}{2}} \left( 2^{\alpha} - 1\right) \left( 2^{1 - \alpha} - 1\right)\right)^{-\frac{2(1 - \alpha)}{(1 - 2\alpha)}} \Lambda_{\alpha}$.
\end{lemma}

We define the following norms on $\mathbb{H}^1_0$, respectively:
\begin{displaymath}
	\mathcal{N}_{g, 0}(h) := \sup\limits_{t \in [0, 1]}\left| \int_{0}^{t} {g(s) h^{\prime}(s)} \,{\rm d}s \right|,
\end{displaymath}
\begin{displaymath}
	\mathcal{N}_{g, \alpha}(h) := \sup\limits_{t \in [0, 1]} \frac{\vert \int_{0}^{t} {g(s) h^{\prime}(s)} \,{\rm d}s - \int_{0}^{r} {g(s) h^{\prime}(s)} \,{\rm d}s \vert}{\vert t - r \vert^{\alpha}}, \quad 0 < \alpha < \frac{1}{4},
\end{displaymath}
\begin{displaymath}
	\mathcal{N}_{g}(h) := \mathcal{N}_{g, 0}(h) + \mathcal{N}_{g, \alpha}(h), \quad 0 < \alpha < \frac{1}{4}.
\end{displaymath}
To apply the technical lemmas established above, we prove the following lemma, which ensures that the norms \(\mathcal{N}_g\) are measurable for \(0 < \alpha < \frac{1}{4}\).
\begin{lemma}\label{lemma 2.8}
	$\mathcal{N}_g$ with $0 < \alpha <\frac{1}{4}$ are measurable norms and we have $\tilde{\mathcal{N}}_g = \Vert  \int_{0}^{t} {g(s)} \,{\rm d}W_s \Vert_{\alpha}$.
\end{lemma}
\begin{proof}
	According to the properties of norm and semi-norm, it suffices to show that $\mathcal{N}_{g, 0}$ is a measurable norm and $\mathcal{N}_{g, \alpha}(h)$ is a measurable semi-norm. Below, we will only prove that $\mathcal{N}_{g, 0}$ is a measurable norm, because the proof of $\mathcal{N}_{g, \alpha}$ is similar to $\mathcal{N}_g$. Fix $t \in [0, 1]$ and define the continuous linear functional $\varphi_t : \mathbb{H}^1_0 \to \mathbb{R}^n$ as
	\begin{displaymath}
		\varphi_t(h) = \int_{0}^{t} { g(s) h^{\prime}(s) } \,{\rm d}s.
	\end{displaymath}
	Then we can show that $\vert \varphi_t(\cdot) \vert$ represents a measurable norm. Define the sequence of measurable norms $\mathcal{N}_n(h) = \sup\limits_{0 \leq j \leq 2^n} \vert \varphi_{j2^{-n}}(h) \vert$. In addition, we have the following convergence regarding limits $n \to \infty$,
	\begin{displaymath}
		\tilde{\mathcal{N}}_n = \sup\limits_{0 \leq j \leq 2^n} \bigg\vert \int_{0}^{j2^{-n}}  {g(s)} \,{\rm d}W_s  \bigg\vert \stackrel{\mathbb{P}}{\longrightarrow} \sup\limits_{0 \leq t \leq 1} \bigg\vert \int_{0}^{t}  {g(s)} \,{\rm d}W_s  \bigg\vert,
	\end{displaymath}
	and by lemma $\ref{lemma 2.7}$, we have
	\begin{displaymath}
		\mathbb{P} \bigg( \sup\limits_{0 \leq t \leq 1} \bigg\vert \int_{0}^{t} {g(s)} \,{\rm d}W_s  \bigg\vert \leq \epsilon \bigg) > 0.
	\end{displaymath}
	According to Lemma $\ref{lemma 2.2}$, $\mathcal{N}_{g, 0} = \lim\limits_{n \to \infty} \mathcal{N}_n(\cdot)$ is a measurable norm. Similarly, it is straightforward to obtain that $\mathcal{N}_{g,\alpha}$ is a measurable semi-norm. Therefore, $\mathcal{N}_g = \lim\limits_{n\rightarrow\infty}\mathcal{N}_n(\cdot)$ is a measurable norm and we have $\tilde{\mathcal{N}}_g = \Vert  \int_{0}^{t} {g(s)} \,{\rm d}W_s \Vert_{\alpha}$.
\end{proof} 

Below is the general lemma of the KAM scheme, for a comprehensive proof, refer to Appendix B of \cite{6}.
\begin{lemma}\label{L2.14}
	Let $r > 0$, $0 < \bar{\sigma} \leq 1$, $0 < 2\sigma < s \leq 1$, and $\mathbb{D} \subset \mathbb{R}^n$ be a non-empty, bounded domain. Consider the Hamiltonian
	\[
	H(I, \theta) := H_0(I) + P(I, \theta),
	\]
	where $H_0, P \in \mathcal{A}_{r, s}(\mathbb{D})$. Assume the following conditions hold:
	\begin{equation}\label{2.1}
		\begin{aligned}
			\det H_{0,II}(I) &\neq 0, & \mathcal{T}(I) &:= H_{0,II}(I)^{-1}, & \forall I \in \mathbb{D}, \\
			\|H_{0,II}\|_{r, \mathbb{D}} &\leq C_H, & \|\mathcal{T}\|_{\mathbb{D}} &\leq C_T, & \\
			\|P\|_{r, s, \mathbb{D}} &\leq \eta, & K_I(\mathbb{D}) &\subset \Delta_{\alpha}^{\tau}, &
		\end{aligned}
	\end{equation}
	where \( H_{0,II}(I) \) denotes the Hessian matrix of \( H_{0}(I) \) with respect to the action variables \( I \). Define the parameters:
	\begin{equation}\label{2.2}
		\begin{aligned}
			\theta &:= C_H C_T, & \lambda &:= \log \rho^{-1}, & \kappa &:= 6 \sigma^{-1} \lambda, \\
			\breve{r} &\leq \frac{r}{32n \theta}, & \overline{r} &\leq \min \left\{ \frac{\alpha}{2n C_H \kappa^{\nu}}, \tilde{r} \right\}, &  \tilde{r}&:= \frac{\breve{r} \bar{\sigma}}{16n \theta}, \\
			\bar{s} &:= s - \frac{2}{3} \sigma, & s' &:= s - \sigma, & \mathrm{L} &:= C_0 \frac{\theta^2 \eta}{r \tilde{r}}.
		\end{aligned}
	\end{equation}
	Furthermore, assume
	\begin{equation}\label{2.3}
		\begin{aligned}
			\sigma^{-\nu} \frac{\eta}{\alpha r} \leq \rho \leq \frac{1}{4}, \qquad r \leq \frac{\alpha}{C_H} \sigma^{\nu},\qquad  \mathrm{L} \leq \frac{\bar{\sigma}}{3}.
		\end{aligned}
	\end{equation}
	Then, there exists a diffeomorphism $G: D_{\bar{r}}(\mathbb{D}) \to G(D_{\bar{r}}(\mathbb{D}))$ and a symplectic change of coordinates $\phi' = \mathrm{Id} + \tilde{\phi}: D_{\bar{r}/2, s'}(\mathbb{D}') \to D_{\bar{r} + r \sigma / 3, \bar{s}}(\mathbb{D})$, where $\mathbb{D}' := G(\mathbb{D})$, such that
	\begin{equation}\label{2.4}
		\begin{cases}
			H \circ \phi' = H' := H_{0}' + P', \\
			\partial_{I'} H_{0}' \circ G = \partial_I H_{0}, \quad \det \partial_{I'}^2 H_{0}' \circ G \neq 0 \quad \text{on } \mathbb{D},
		\end{cases}
	\end{equation}
	where $H_{0}'(I') := H_{0}(I') + \tilde{H_{0}}(I') := H_{0}(I') + \langle P(I', \cdot) \rangle$, and $G = (\partial_{I'} H_{0}')^{-1} \circ H_{0,I}$. Additionally, setting $(\partial_{I'}^2 H_{0}'(I'))^{-1} := \mathcal{T}(I') + \tilde{\mathcal{T}}(I')$ for $I' \in \mathbb{D}'$, the following estimates hold:
	\begin{equation}\label{2.5}
		\begin{aligned}
			&\|\partial_{I'}^2 \tilde{H_{0}}\|_{\bar{r}/2, \mathbb{D}'} \leq C_H \mathrm{L},\qquad  \|G - \text{id}\|_{\bar{r}, \mathbb{D}} \leq \bar{r} \mathrm{L}, \qquad  \|\tilde{\mathcal{T}}\|_{\mathbb{D}'} \leq C_T \mathrm{L}, \\
			&\max \left\{ \|M \tilde{\phi}\|_{\bar{r}/2, s', \mathbb{D}'}, \|\pi_2 \partial_{\theta'} \tilde{\phi}\|_{\bar{r}/2, s', \mathbb{D}'} \right\} \leq C_1 \frac{\eta}{\alpha r \sigma^{\nu}}, \quad  \|P'\|_{\bar{r}/2, s', \mathbb{D}'} \leq C_1 \rho \eta,
		\end{aligned}
	\end{equation}
	with $M := \text{diag}(r^{-1} \mathbb{I}_n, \sigma^{-1} \mathbb{I}_n)$.
\end{lemma}

For the approximation of smooth functions using real-analytic functions and the uniform convergence of sequences of real-analytic functions, we refer to the relevant results in \cite{48}.
\begin{lemma}\label{L2.15}
	Let \(l > 0\). There exists a constant \(C = C(n, l)>0\) such that for any \(f\in C^l(\mathbb{R}^n \times \mathbb{T}^n)\) and \(s > 0\), there is a real-analytic function \(f_s:\Omega_s\to\mathbb{C}\) defined on the complex domain
	\[
	\Omega_s := \left\{ (I, \theta) \in \mathbb{C}^n \times \mathbb{C}^n \ \big| \ \max\{|\textup{Im}\, I|, |\textup{Im}\, \theta|\} < s \right\},
	\]
	satisfying the following:
	\begin{enumerate}
		\item Uniform bound: 
		\[
		\sup_{\Omega_s} |f_s| \leq C \|f\|_{C^0}.
		\]
		\item Approximation error: For any integer \(0\leq l'\leq l\), 
		\[
		\|f - f_s\|_{C^{l'}} \leq C \|f\|_{C^l} \cdot s^{l - l'}.
		\]
		\item Derivative stability: For any \(0 < s' < s\) and multi-index \(\alpha\) with \(|\alpha| \leq l'\), 
		\[
		\sup_{\Omega_{s'}} \left|\partial^\alpha f_s - \partial^\alpha f_{s'}\right| \leq C \|f\|_{C^l} s^{l - l'}.
		\]
	\end{enumerate}
	Moreover, if \(f\) is periodic in a component \(I_i\) or \(\theta_i\), then \(f_s\) preserves periodicity in that component.
\end{lemma}
\begin{lemma}\label{L2.16}
	Let \(\{f_j\}_{j\geq0}\) be a sequence of real-analytic functions defined on nested domains
	\[ 
	\Omega_j := \left\{ (I, \theta) \in \mathbb{C}^n \times \mathbb{C}^n \ \big| \ |\textup{Im}(I, \theta)| < s_j \right\},
	\]
	where \(s_j = s_0 \kappa^j\) for \(s_0 > 0\), \(0 < \kappa < 1\), and \(l \in \mathbb{R}^+\setminus\mathbb{Z}\). Suppose the sequence satisfies
	\[
	\sup_{\Omega_j} |f_j - f_{j - 1}| \leq \Gamma \cdot s_{j - 1}^l \quad \text{for all } j\geq1,
	\]
	where \(\Gamma > 0\). Then:
	\begin{enumerate}
		\item Uniform convergence: \(f_j\) converges uniformly on \(\mathbb{R}^n \times \mathbb{T}^n\) to a limit \(f\in C^l(\mathbb{R}^n \times \mathbb{T}^n)\).
		\item Periodicity preservation: If all \(f_j\) are periodic in a component \(I_i\) or \(\theta_i\), the limit \(f\) inherits periodicity in that component.
	\end{enumerate}
\end{lemma}

\section{Onsager-Machlup functional for stochastic Hamiltonian systems}
In this section, we derive the Onsager-Machlup functional for stochastic Hamiltonian systems by analyzing the asymptotic behavior of the probability ratio that a stochastic trajectory lies in a small neighborhood of a prescribed reference path. In Sections 3, 4, and 5, we consider a general stochastic Hamiltonian system of the following form
\begin{equation}\label{3.1}
	\begin{cases}
		\,{\rm d} q(t) = \frac{\partial H}{\partial p}(q(t), p(t)) \,{\rm d}t + \sigma_q(t) \,{\rm d} W_q(t),\\
		\,{\rm d} p(t) = -\frac{\partial H}{\partial q}(q(t), p(t)) \,{\rm d}t + \sigma_p(t) \,{\rm d} W_p(t),
	\end{cases}
\end{equation}
with initial condition $\left( q(0) , p(0)\right)  = \left( q_0 , p_0\right)$. Here  $ H(q, p) $ is the Hamiltonian, $ \sigma_q(t) $ and $ \sigma_p(t) $ represent the strengths of the stochastic perturbations, and $ W_q(t) $ and $ W_p(t) $ are standard Wiener processes in $\mathbb{R}^n$.

In Sections 3-5, we work with stochastic Hamiltonian systems expressed in the position-momentum coordinates $(q,p)$. The change of variables from $(q,p)$ to the action-angle variables $(I,\theta)$ can be obtained via a Legendre-type transformation and does not create any compatibility issues for our analysis. We adopt the $(q,p)$ formulation primarily because it simplifies the proofs. Moreover, it provides a more natural framework for modeling physical stochastic perturbations, since the noise typically acts directly on the position (e.g., thermal displacement fluctuations) or on the momentum (e.g., random impulsive velocity kicks). Our results remain valid on any finite time interval $[0,T]$. For notational convenience, however, we restrict the exposition to the normalized interval $[0,1]$ throughout the sequel.

We begin by stating the conditions on the functions \(H(\theta, I)\), \(\sigma_\theta(t)\), and \(\sigma_I(t)\) that will be assumed throughout the proof:

\begin{itemize}
	\item[(C1)]The Hamiltonian function $ H(q, p) \in C^3_b(\mathbb{R}^n \times \mathbb{R}^n, \mathbb{R}) $, means that $ H $ is three times continuously differentiable with bounded third order derivatives. Additionally, the partial derivatives $ \frac{\partial H}{\partial q} $ and $ \frac{\partial H}{\partial p} $ are globally Lipschitz continuous.
	\item[(C2)] The diffusion matrices $\sigma_q, \sigma_p \in C\!\big([0,1];\mathbb{R}^{n\times n}\big)$ are diagonal. Writing
	$\sigma_q(t)=\mathrm{diag}\big(\sigma_{q,1}(t),\dots,\sigma_{q,n}(t)\big)$ and
	$\sigma_p(t)=\mathrm{diag}\big(\sigma_{p,1}(t),\dots,\sigma_{p,n}(t)\big)$, suppose there exist $0<\lambda\le \Lambda<\infty$ such that for all $t\in[0,1]$ and $1\le i\le n$,
	\[
	\lambda \le \sigma_{q,i}(t) \le \Lambda,
	\qquad
	\lambda \le \sigma_{p,i}(t) \le \Lambda .
	\]
\end{itemize}

\begin{theorem}	\label{T3.1}
	Assume that $(q(t),p(t))$ is a solution of equation $\eqref{3.1}$, the reference path $\varphi(t) := \left( \varphi_q(t),\varphi_p(t)\right)$ is a function such that $(\varphi_q(0),\varphi_p(0)) = (q_0,p_0) $ and $\left( (\varphi_q(t),\varphi_p(t)) - (q_0,p_0) \right)$ belongs to Cameron-Martin $\mathbb{H}^1_0$. And assume that $\sigma_q(t)$, $\sigma_p(t)$ and $H(q,p)$ satisfy Conditions $(C1)$ and $(C2)$. If we use the Hölder norm $\Vert \cdot \Vert_{\alpha}$ with $ 0< \alpha < \frac{1}{4}$, then the Onsager-Machlup functional of solution process $(q(t),p(t))$ exists and has the form
	\begin{equation}\label{3.2}
		\int_{0}^{1} OM(\varphi_q, \varphi_p) \,{\rm d}t = \int_0^1 \left| \sigma_q^{-1}(t) \left( \dot{\varphi}_q - \frac{\partial H}{\partial \varphi_p}(\varphi_q, \varphi_p) \right) \right|^2 \,{\rm d}t + \int_0^1 \left| \sigma_p^{-1}(t) \left( \dot{\varphi}_p + \frac{\partial H}{\partial \varphi_q}(\varphi_q, \varphi_p) \right) \right|^2 \,{\rm d}t.
	\end{equation}
	where $\dot{\varphi}(t) = \bigl(\dot{\varphi}_q(t),\dot{\varphi}_p(t)\bigr)$, $\sigma_{q}^{-1}(t)$ and $\sigma_{p}^{-1}(t)$ denote the inverses of $\sigma_{q}(t)$ and $\sigma_{p}(t)$, respectively.
\end{theorem}
\begin{proof}
	Let the reference path be given by $\varphi(t) = (\varphi_q(t), \varphi_p(t))$, where $\varphi(t)$ is a definite continuous path, and $(\varphi_q(t), \varphi_p(t)) - (q_0,p_0) \in \mathbb{H}^1_0$. We define the perturbed solution, denoted as $(y_q(t), y_p(t))$, as follows:
	\begin{equation}\label{3.3}
		\begin{cases}
			y_q(t) = \varphi_q(t) + \int_0^t \sigma_q(s) \,{\rm d}W_q(s),\\
			y_p(t) = \varphi_p(t) + \int_0^t \sigma_p(s) \,{\rm d}W_p(s).
		\end{cases}
	\end{equation}
	To simplify the notation in the proof, we introduce the term $ W^{\sigma}(t) := \left( W^{\sigma}_q(t), W^{\sigma}_p(t) \right) $, which represents the stochastic perturbation in the system
	\begin{displaymath}
		W^{\sigma}_q(t) := \int_0^t \sigma_q(s) \,{\rm d}W_q(s), \quad W^{\sigma}_p(t) := \int_0^t \sigma_p(s) \,{\rm d}W_p(s).
	\end{displaymath}
	
	We define \( \tilde{W}_q(t) \) and \( \tilde{W}_p(t) \) as follows. It can be shown that under the new probability measures \( \tilde{\mathbb{P}}_q \) and \( \tilde{\mathbb{P}}_p \), \( \tilde{W}_q(t) \) and \( \tilde{W}_p(t) \) are standard Brownian motions.
	\begin{equation}\label{3.4}
		\begin{aligned}
			\tilde{W}_q(t) &= W_q(t) - \int_0^t \sigma_q^{-1}(s) \left( \frac{\partial H}{\partial y_p}(y_q, y_p) - \dot{\varphi}_q(s) \right) \,{\rm d}s,\\
			\tilde{W}_p(t) &= W_p(t) - \int_0^t \sigma_p^{-1}(s) \left( -\frac{\partial H}{\partial y_q}(y_q, y_p) - \dot{\varphi}_p(s) \right) \,{\rm d}s.
		\end{aligned}
	\end{equation}
	Substituting the Brownian motions defined in equation $\eqref{3.4}$ into equation $\eqref{3.3}$, we obtain
	\begin{equation}\label{3.5}
		\begin{cases}
			\,{\rm d} y_q(t) = \frac{\partial H}{\partial y_p}(y_q, y_p) dt + \sigma_q(t) \,{\rm d}\tilde{W}_q(t),\\
			\,{\rm d} y_p(t) = -\frac{\partial H}{\partial y_q}(y_q, y_p) dt + \sigma_p(t) \,{\rm d}\tilde{W}_p(t).
		\end{cases}
	\end{equation}
	It can be observed that under the new probability measure $\tilde{\mathbb{P}}= \tilde{\mathbb{P}}_q \otimes \tilde{\mathbb{P}}_p $, $ (y_q(t), y_p(t)) $ is a solution to the equation $\eqref{3.1}$.
	
	To apply Girsanov's Theorem and achieve the transformation between the two measures, we define the Radon-Nikodym derivative \(\mathcal{R} := \frac{d\tilde{\mathbb{P}}}{d\mathbb{P}} = \frac{{\rm d}\tilde{\mathbb{P}}_q}{{\rm d}\mathbb{P}_q} \cdot \frac{{\rm d}\tilde{\mathbb{P}}_p}{{\rm d}\mathbb{P}_p}\), which represents the change of measure from \(\mathbb{P}\) to \(\tilde{\mathbb{P}}\). This derivative is given by an exponential martingale associated with the drift terms, which describes the behavior of the Brownian motion under the new measure after the removal of the drift. For the angle variables \(q\), the Radon-Nikodym derivative is
	\begin{displaymath}
		\begin{aligned}
			\frac{d\tilde{\mathbb{P}_q}}{d\mathbb{P}_q} = \exp &\left( \int_0^1 \left\langle \sigma_q^{-1}(s) \left( \frac{\partial H}{\partial y_p}(y_q, y_p) - \dot{\varphi}_q(s) \right), \,{\rm d}W_q(s) \right\rangle\right. \\
			&\left. - \frac{1}{2} \int_0^1 \left| \sigma_q^{-1}(s) \left( \frac{\partial H}{\partial y_p}(y_q, y_p) - \dot{\varphi}_q(s) \right) \right|^2 \,{\rm d}s \right),
		\end{aligned}
	\end{displaymath}
	and similarly for the action variables \(p\)
	\begin{displaymath}
		\begin{aligned}
			\frac{d\tilde{\mathbb{P}}_p}{d\mathbb{P}_p} = \exp &\left( \int_0^1 \left\langle \sigma_p^{-1}(s) \left( -\frac{\partial H}{\partial y_q}(y_q, y_p) - \dot{\varphi}_p(s) \right), \,{\rm d}W_p(s) \right\rangle\right. \\
			&\left. - \frac{1}{2} \int_0^1 \left| \sigma_p^{-1}(s) \left( \frac{\partial H}{\partial y_q}(y_q, y_p) + \dot{\varphi}_p(s) \right) \right|^2 \,{\rm d}s \right).
		\end{aligned}
	\end{displaymath}
	So,
	\begin{displaymath}
		\begin{aligned}
			\mathcal{R} = \exp &\left( \int_0^1 \left\langle \sigma_q^{-1}(s) \left( \frac{\partial H}{\partial y_p}(y_q, y_p) - \dot{\varphi}_q(s) \right), \,{\rm d}W_q(s) \right\rangle\right. \\
			&- \int_0^1 \left\langle \sigma_p^{-1}(s) \left( \frac{\partial H}{\partial y_q}(y_q, y_p) + \dot{\varphi}_p(s) \right), \,{\rm d}W_p(s) \right\rangle
			\\ & - \frac{1}{2} \int_0^1 \left| \sigma_q^{-1}(s) \left( \frac{\partial H}{\partial y_p}(y_q, y_p) - \dot{\varphi}_q(s) \right) \right|^2 \,{\rm d}s\\
			&\left. - \frac{1}{2} \int_0^1 \left| \sigma_p^{-1}(s) \left( \frac{\partial H}{\partial y_q}(y_q, y_p) + \dot{\varphi}_p(s) \right) \right|^2 \,{\rm d}s \right).
		\end{aligned}
	\end{displaymath}
	We now aim to compute the transition probability of the system's path remaining close to the reference path \(\varphi(t)\). Using Girsanov's Theorem, this probability can be expressed as
	\begin{equation}\label{3.6}
		\begin{aligned}
			&\quad \frac{\mathbb{P}\left(\left\|  (q, p) - (\varphi_q, \varphi_p) \right\|_{\alpha}  \leq \epsilon\right)}{\mathbb{P}\left(\left\|  W^{\sigma} \right\|_{\alpha}  \leq \epsilon\right)}
			= \frac{\tilde{\mathbb{P}}\left(\left\| (Y_q, Y_p) -(\varphi_q, \varphi_p) \right\|_{\alpha} \leq \epsilon\right)}{\mathbb{P}\left(\left\| W^{\sigma} \right\|_{\alpha} \leq \epsilon\right)}\\
			&= \frac{\tilde{\mathbb{P}}\left(\left\| W^{\sigma} \right\|_{\alpha} \leq \epsilon\right)}{\mathbb{P}\left( \left\| W^{\sigma} \right\|_{\alpha} \leq \epsilon\right)}
			= \frac{\mathbb{E} \left( \mathcal{R}\mathbb{I}_{ \left\| W^{\sigma} \right\|_{\alpha} \leq \epsilon} \right)}{\mathbb{P}\left( \left\| W^{\sigma} \right\|_{\alpha} \leq \epsilon\right)}
			= \mathbb{E}\left( \mathcal{R} ~\big|~ \left\| W^{\sigma} \right\|_{\alpha} \leq \epsilon \right)\\
			& = \exp\left\lbrace -\frac{1}{2} \left(  \int_0^1 \left| \sigma_q^{-1}(t) \left( \dot{\varphi}_q - \frac{\partial H}{\partial \varphi_p}(\varphi_q, \varphi_p) \right) \right|^2 \,{\rm d}t + \int_0^1 \left| \sigma_p^{-1}(t) \left( \dot{\varphi}_p + \frac{\partial H}{\partial \varphi_q}(\varphi_q, \varphi_p) \right) \right|^2 \,{\rm d}t \right) \right\rbrace 
			\\& \quad \times \mathbb{E} \left( \exp\left\lbrace  \sum_{i=1}^{6} B_i \right\rbrace ~\big|~ \left\| W^{\sigma} \right\|_{\alpha} \leq \epsilon \right),
		\end{aligned}
	\end{equation}
	where $B_i$ represents the deviations in the path arising from drift and perturbations, it exhibits stochastic properties. This is further clarified by the following detailed expression:
	\begin{displaymath}
		\begin{aligned}
			B_1 &= \int_0^1 \left\langle \sigma_q^{-1}(s) \frac{\partial H}{\partial y_p}(y_q, y_p) , \,{\rm d}W_q(s) \right\rangle - \int_0^1 \left\langle \sigma_p^{-1}(s)  \frac{\partial H}{\partial y_q}(y_q, y_p), \,{\rm d}W_p(s) \right\rangle,\\
			B_2 &= - \int_0^1 \left\langle \sigma_q^{-1}(s) \dot{\varphi}_q(s), \,{\rm d}W_q(s) \right\rangle - \int_0^1 \left\langle \sigma_p^{-1}(s) \dot{\varphi}_p(s), \,{\rm d}W_p(s) \right\rangle,\\
			B_3 &= \frac{1}{2} \int_0^1 \left| \sigma_q^{-1}(s) \frac{\partial H}{\partial \varphi_p}(\varphi_q, \varphi_p)  \right|^2 \,{\rm d}s - \frac{1}{2} \int_0^1 \left| \sigma_q^{-1}(s)  \frac{\partial H}{\partial y_p}(y_q, y_p) \right|^2 \,{\rm d}s,\\
			B_4 &= \frac{1}{2} \int_0^1 \left| \sigma_p^{-1}(s) \frac{\partial H}{\partial \varphi_q}(\varphi_q, \varphi_p)  \right|^2 \,{\rm d}s - \frac{1}{2} \int_0^1 \left| \sigma_p^{-1}(s)  \frac{\partial H}{\partial y_q}(y_q, y_p) \right|^2 \,{\rm d}s,\\
			B_5 &= \int_0^1 \left\langle \sigma_q^{-1}(s) \left( \frac{\partial H}{\partial y_p}(y_q, y_p) - \frac{\partial H}{\partial \varphi_p}(\varphi_q, \varphi_p) \right), \sigma_q^{-1}(s) \dot{\varphi}_q(s) \right\rangle \,{\rm d}s,\\
			B_6 &= - \int_0^1 \left\langle \sigma_p^{-1}(s) \left( \frac{\partial H}{\partial y_q}(y_q, y_p) - \frac{\partial H}{\partial \varphi_q}(\varphi_q, \varphi_p) \right), \sigma_p^{-1}(s) \dot{\varphi}_p(s) \right\rangle \,{\rm d}s.
		\end{aligned}
	\end{displaymath}
	For the second term $B_2$, we have
	\begin{equation}
		\begin{aligned}
			B_2 &= - \int_0^1 \left\langle \sigma_q^{-1}(s) \dot{\varphi}_q(s), \,{\rm d}W_q(s) \right\rangle - \int_0^1 \left\langle \sigma_p^{-1}(s) \dot{\varphi}_p(s), \,{\rm d}W_p(s) \right\rangle
			\\&= - \sum_{i=1}^n \left( \int_0^1 \left[ \sigma_q^{-1}(s) \dot{\varphi}_q(s)\right]_{i}  \,{\rm d}W_{q,i}(s) + \int_0^1 \left[ \sigma_p^{-1}(s) \dot{\varphi}_p(s)\right]_{i}  \,{\rm d}W_{p,i}(s) \right).
		\end{aligned}\notag
	\end{equation}
	It is straightforward to demonstrate that $\sigma_q^{-1}(s) \dot{\varphi}_q(s) \in L^2\left( [0,1] ; \mathbb{R}^n \right) $ and $\sigma_p^{-1}(s) \dot{\varphi}_p(s) \in L^2\left( [0,1] ; \mathbb{R}^n \right)$. Applying Lemma $\ref{lemma 2.4}$ and Lemma $\ref{lemma 2.8}$, we subsequently obtain
	\begin{equation}\label{3.7}
		\limsup\limits_{\epsilon \to 0} \mathbb{E}\left({\rm exp}\left\{ cB_2 \right\} ~\big|~\left\| W^{\sigma} \right\|_{\alpha} < \epsilon \right) = 1
	\end{equation}
	for all $c \in \mathbb{R}$.
	
	For the third term $B_3$, 
	\begin{displaymath}
		\begin{aligned}
			B_3 &= \frac{1}{2} \int_0^1 \left| \sigma_q^{-1}(s) \frac{\partial H}{\partial \varphi_p}(\varphi_q, \varphi_p)  \right|^2 \,{\rm d}s - \frac{1}{2} \int_0^1 \left| \sigma_q^{-1}(s)  \frac{\partial H}{\partial y_p}(y_q, y_p) \right| ^2 \,{\rm d}s
			\\ &\leq \frac{1}{2} \int_{0}^{1} {\left| \left\| \sigma_q^{-1}(s)  \right\|\right|^2 \left|  \frac{\partial H}{\partial \varphi_p}(\varphi_q, \varphi_p) - \frac{\partial H}{\partial y_p}(y_q, y_p) \right|^2 \,{\rm d}s}\\
			& \quad + \int_{0}^{1} {\left| \left\| \sigma_q^{-1}(s)  \right\|\right|^2 \left|\frac{\partial H}{\partial \varphi_p}(\varphi_q, \varphi_p) - \frac{\partial H}{\partial y_p}(y_q, y_p) \right| \left| \frac{\partial H}{\partial y_p}(y_q, y_p) \right|} \,{\rm d}s.
		\end{aligned}
	\end{displaymath}
	In Condition (C1), since \(\frac{\partial H}{\partial p}\) is Lipschitz continuous, we have the following estimate
	\begin{equation}\label{5}
		\begin{aligned}
			& \quad \left| \frac{\partial H}{\partial y_p}(y_q, y_p)  -  \frac{\partial H}{\partial \varphi_p}(\varphi_q, \varphi_p)\right|\\
			&= \left| \frac{\partial H}{\partial \left( \varphi_p + W^{\sigma}_p \right)}((\varphi_q + W^{\sigma}_q), (\varphi_p + W^{\sigma}_p))-  \frac{\partial H}{\partial \varphi_p}(\varphi_q, \varphi_p)\right|
			\leq L \left| W^{\sigma} \right|.
		\end{aligned}
	\end{equation}
	Inequality $\eqref{5}$ and the boundedness of $\frac{\partial H}{\partial y_p}(y_q, y_p)$ and $\sigma_q^{-1}(t)$ imply that
	\begin{equation}\label{3.9}
		\limsup\limits_{\epsilon \to 0} \mathbb{E}\left({\rm exp}\left\{ cB_3 \right\} ~\big|~\left\| W^{\sigma} \right\|_{\alpha} < \epsilon \right) = 1
	\end{equation}
	for all $c \in \mathbb{R}$.
	
	For the fourth term $B_4$, employing the same proof technique as for the third term $B_3$, we have
	\begin{equation}\label{3.10}
		\limsup\limits_{\epsilon \to 0} \mathbb{E}\left({\rm exp}\left\{ cB_4 \right\} ~\big|~\left\| W^{\sigma} \right\|_{\alpha} < \epsilon \right) = 1
	\end{equation}
	for all $c \in \mathbb{R}$.
	
	For the fifth term $B_5$, applying inequality $\eqref{5}$ and the boundedness of $\dot{\varphi}_q(t)$ and $\sigma_q^{-1}(t)$, we have
	\begin{displaymath}
		\begin{aligned}
			B_5 &= \int_0^1 \left\langle \sigma_q^{-1}(s) \left( \frac{\partial H}{\partial y_p}(y_q, y_p) - \frac{\partial H}{\partial \varphi_p}(\varphi_q, \varphi_p) \right), \sigma_q^{-1}(s) \dot{\varphi}_q(s) \right\rangle \,{\rm d}s\\
			& \leq C \int_0^1 \left| \left\| \sigma_q^{-1}(s)  \right\|\right|^2 \left|  \frac{\partial H}{\partial y_p}(y_q, y_p) - \frac{\partial H}{\partial \varphi_p}(\varphi_q, \varphi_p) \right| \left| \dot{\varphi}_q(s) \right| \,{\rm d}s\\
			& \leq CL \left\| W^{\sigma} \right\|_{\alpha}.
		\end{aligned}
	\end{displaymath}
	Thus,
	\begin{equation}\label{3.11}
		\limsup\limits_{\epsilon \to 0} \mathbb{E}\left({\rm exp}\left\{ cB_5 \right\} ~\big|~\left\|  W^{\sigma} \right\|_{\alpha}  < \epsilon \right) = 1
	\end{equation}
	for all $c \in \mathbb{R}$.
	
	For the sixth term $B_6$, employing the same proof technique as for the fifth term $B_5$, we have
	\begin{equation}\label{3.12}
		\limsup\limits_{\epsilon \to 0} \mathbb{E}\left({\rm exp}\left\{ cB_6 \right\} ~\big|~\left\|  W^{\sigma} \right\|_{\alpha} < \epsilon \right) = 1
	\end{equation}
	for all $c \in \mathbb{R}$.
	
	For the first term $B_1$, in order to write it as a whole, we define
	\begin{displaymath}
		\begin{aligned}
			\sigma^{-1}(t) &:= \begin{bmatrix}\sigma_q^{-1}(t) & 0 \\ 0 & \sigma_p^{-1}(t)\end{bmatrix}, &
			J \nabla H(y) &:= \begin{pmatrix}\frac{\partial H}{\partial y_p}(y_q, y_p) \\[6pt] - \frac{\partial H}{\partial y_q}(y_q, y_p)\end{pmatrix}, \\
			W(t) &:= \begin{pmatrix} W_q(t) \\ W_p(t) \end{pmatrix}, &
			\mathrm{d}W(t) &:= \begin{pmatrix} \mathrm{d}W_q(t) \\ \mathrm{d}W_p(t) \end{pmatrix}.
		\end{aligned}
	\end{displaymath}
	Under the assumption of small perturbations, it is feasible to apply a Taylor series expansion to $J \nabla H(y)$. Specifically, we have
	\begin{displaymath}
		\begin{aligned}
			J \nabla H(y) &= \begin{pmatrix}\frac{\partial H}{\partial \varphi_p}(\varphi_q, \varphi_p) \\ - \frac{\partial H}{\partial \varphi_q}(\varphi_q, \varphi_p)\\ \end{pmatrix} + \begin{bmatrix}\frac{\partial^2H}{\partial \varphi_q \partial \varphi_p}(\varphi_q, \varphi_p) & \frac{\partial^2H}{\partial \varphi_p^2}(\varphi_q, \varphi_p) \\ - \frac{\partial^2H}{\partial \varphi_q^2}(\varphi_q, \varphi_p) & - \frac{\partial^2H}{\partial \varphi_p \partial \varphi_q}(\varphi_q, \varphi_p) \\ \end{bmatrix} \begin{pmatrix} W^{\sigma}_q(t) \\ W^{\sigma}_p(t)  \\ \end{pmatrix}
			+  \begin{pmatrix}R_q(t) \\ R_p(t)  \\ \end{pmatrix}
			\\& := J \nabla H(\varphi) + J \nabla^2 H(\varphi) W^{\sigma} + R(t),
		\end{aligned}
	\end{displaymath}
	where $J$ denotes the standard symplectic matrix, $\nabla$ is the gradient operator, $\nabla H$ denotes the gradient of $H$, and $\nabla^{2}H$ denotes the Hessian matrix of $H$. According to the properties of the Taylor expansion, when $H \in C^3_b(\mathbb{R}^n \times \mathbb{R}^n, \mathbb{R})$ and $\left\|  W^{\sigma} \right\|_{\alpha}  \leq \epsilon$, we can estimate the remainder term $R(t)$ as
	\begin{displaymath}
		\sup\limits_{0 \leq t \leq 1}\left|  R(t) \right| \leq k \epsilon^2.
	\end{displaymath}
	Hence, $B_1$ can be written as
	\begin{displaymath}
		\begin{aligned}
			B_1 &= \int_0^1 \left\langle \sigma_q^{-1}(s) \frac{\partial H}{\partial y_p}(y_q, y_p) , \,{\rm d}W_q(s) \right\rangle - \int_0^1 \left\langle \sigma_p^{-1}(s)  \frac{\partial H}{\partial y_q}(y_q, y_p), \,{\rm d}W_p(s) \right\rangle\\
			&= \int_0^1 \left\langle \sigma^{-1}(s)  J \nabla H(y),  \,{\rm d}W(s)\right\rangle \\
			&= \int_0^1 \left\langle \sigma^{-1}(s)  J \nabla H(\varphi),  \,{\rm d}W(s)
			\right\rangle + \int_0^1 \left\langle \sigma^{-1}(s) J \nabla^2 H(\varphi) W^{\sigma},  \,{\rm d}W(s)\right\rangle\\
			&\quad + \int_0^1 \left\langle \sigma^{-1}(s)   R(s),  \,{\rm d}W(s)\right\rangle 
			\\ &:= B_{11} + B_{12} + B_{13}.
		\end{aligned}
	\end{displaymath}
	The term $B_{11}$ has the same expression as $B_2$
	\begin{displaymath}
		\begin{aligned}
			B_{11} =  \int_0^1 \left\langle \sigma^{-1}(s)  J \nabla H(\varphi),  \,{\rm d}W(s)\right\rangle.
		\end{aligned}
	\end{displaymath}
	Due to $H \in C^3_b(\mathbb{R}^n \times \mathbb{R}^n, \mathbb{R})$, we can show that $\sigma^{-1}(s)  J \nabla H(\varphi) \in L^2([0,1], \mathbb{R}^{2n})$. Using the same method as item $B_2$ yields
	\begin{equation}\label{3.13}
		\limsup\limits_{\epsilon \to 0} \mathbb{E}\left({\rm exp}\left\{ cB_{11} \right\} ~\big|~\left\|  W^{\sigma} \right\|_{\alpha} < \epsilon \right) = 1
	\end{equation}
	for all $c \in \mathbb{R}$.
	In order to apply Lemma $\ref{lemma 2.5}$, we will express the term $B_{12}$ as a double stochastic integral with respect to $W$. We have
	\begin{displaymath}
		\begin{aligned}
			B_{12} & = \int_0^1  \left\langle \sigma^{-1}(s)  J \nabla^2 H(\varphi) W^{\sigma},  \,{\rm d}W(s)\right\rangle
			\\ &= \int_{0}^{1} \int_{0}^{s}  \left\langle \sigma^{-1}(s)  J \nabla^2 H(\varphi) \sigma(t),  \,{\rm d}W(t)\right\rangle \,{\rm d}W(s)
			\\ &= \int_{0}^{1}  \int_{0}^{1} \sigma^{-1}(s) J \nabla^2 H(\varphi) \sigma(t) 1_{t \leq s} \,{\rm d}W(t) \,{\rm d}W(s),
		\end{aligned}
	\end{displaymath}
	where $1_{t \leq s}$ is an indicator function. Define
	\begin{displaymath}
		F(s,t) := \sigma^{-1}(s) J \nabla^2 H(\varphi) \sigma(t) 1_{t \leq s}.
	\end{displaymath}
	Hence, $B_{12} = I_2(\tilde{F})$, where $\tilde{F} := \frac{1}{2} (F + F^{*})$ is the symmetrization of $F$. According to Conditions (C1) and (C2), the operator $\mathcal{S}(\tilde{F})$ is nuclear, and its trace can be computed as
	\begin{displaymath}
		\mathrm{Tr}~\tilde{F} = \mathrm{Tr}~{F} = \frac{1}{2} \int_{0}^{1} {\mathrm{Tr}~ {F}(t, t) } \,{\rm d}t = \frac{1}{2}  \int_{0}^{1} \mathrm{Tr}~\left(  J \nabla^2 H(\varphi) \right)  \,{\rm d}t.
	\end{displaymath}
	In addition, due to $H \in C^3_b(\mathbb{R}^n \times \mathbb{R}^n, \mathbb{R})$, we have
	\begin{displaymath}
		\mathrm{Tr}~\left( \frac{\partial^2H}{\partial \varphi_q \partial \varphi_p}(\varphi_q, \varphi_p) \right) = \mathrm{Tr}~\left( \frac{\partial^2H}{\partial \varphi_p \partial \varphi_q}(\varphi_q, \varphi_p) \right),
	\end{displaymath}
	and since the Hamiltonian equations possess a symplectic structure, we obtain
	\begin{displaymath}
		\mathrm{Tr}~\left(  J \nabla^2 H(\varphi) \right) = \mathrm{Tr}~\left( \frac{\partial^2H}{\partial \varphi_q \partial \varphi_p}(\varphi_q, \varphi_p) \right) - \mathrm{Tr}~\left( \frac{\partial^2H}{\partial \varphi_p \partial \varphi_q}(\varphi_q, \varphi_p) \right) = 0.
	\end{displaymath}
	By Lemma $\ref{lemma 2.5}$ and Lemma $\ref{lemma 2.8}$, we have
	\begin{equation}\label{3.14}
		\limsup\limits_{\epsilon \to 0} \mathbb{E}\left({\rm exp}\left\{ cB_{12} \right\} ~\big|~\left\|  W^{\sigma} \right\|_{\alpha} < \epsilon \right) = 1
	\end{equation}
	for all $c \in \mathbb{R}$.
	Finally, we study the behaviour of the term $B_{13}$. For any $c \in \mathbb{R}$ and $\delta > 0$, we have
	\begin{equation}\label{12}
		\begin{aligned}
			&\quad \mathbb{E} \left( {\rm exp} \left\{ cB_{13} \right\}~\big|~\left\|  W^{\sigma} \right\|_{\alpha} \leq \epsilon \right)\\
			&= \int_{0}^{\infty} {e^x \mathbb{P}\left( \left| c\int_0^1 \left\langle \sigma^{-1}(s) R(s),  \,{\rm d}W(s) \right\rangle \right| > x ~\big|~ \left\|  W^{\sigma} \right\|_{\alpha} \leq \epsilon \right)} \,{\rm d}x
			\\ &\leq e^{\delta} + \int_{\delta}^{\infty} {e^x \mathbb{P}\left( \left| c\int_{0}^{1} \left\langle \sigma^{-1}(s) R(s),  \,{\rm d}W(s) \right\rangle \right| > x ~\big|~ \left\|  W^{\sigma} \right\|_{\alpha} \leq \epsilon \right)} \,{\rm d}x.
		\end{aligned}
	\end{equation}
	Define the martingale $M_t = c\int_{0}^{t} \langle \sigma^{-1}(s) R(s), \,{\rm d}W(s) \rangle$. We have the estimate about its quadratic variation
	\begin{displaymath}
		\langle M_t \rangle = c^2\int_{0}^{t} {\left| \sigma^{-1}(s) R(s)\right|^2} \,{\rm d}s \leq C \epsilon^4
	\end{displaymath}
	for some $C > 0$. Using the exponential inequality for martingales, we obtain
	\begin{displaymath}
		\mathbb{P}\left( \left| c\int_{0}^{1} \left\langle \sigma^{-1}(s) R(s),  \,{\rm d}W(s) \right\rangle \right| > x, \left\|  W^{\sigma} \right\|_{\alpha} \leq \epsilon \right) \leq {\rm exp}\left\{ -\frac{x^2}{2c\epsilon^4} \right\}.
	\end{displaymath}
	Then, by Lemma $\ref{lemma 2.6}$, we have
	\begin{displaymath}
		\begin{aligned}
			&\quad \mathbb{P}\left( \left| c\int_{0}^{1} \left\langle \sigma^{-1}(s) 	R(s),  \,{\rm d}W(s) \right\rangle \right| > x ~\big|~ \left\|  W^{\sigma} \right\|_{\alpha} \leq \epsilon \right)\\
			&\leq {\rm exp}\left\{ -\frac{x^2}{2c\epsilon^4} \right\}
			{\rm exp}\left\{ c \left( \frac{\epsilon}{M} \right)^{-\frac{2}{1 - 2\alpha}} \right\}.
		\end{aligned}
	\end{displaymath}
	According to the latter estimate and taking limits in $\eqref{12}$, if $0 < \alpha < \frac{1}{4}$, then
	\begin{displaymath}
		\begin{aligned}
			-\epsilon^{-4} + \epsilon^{-\frac{2}{1 - 2\alpha}} \to - \infty \qquad \text{as}~\epsilon \to 0 .
		\end{aligned}
	\end{displaymath}
	So we have
	\begin{equation}\label{3.16}
		\limsup\limits_{\epsilon \to 0} \mathbb{E}\left({\rm exp}\left\{ cB_{13} \right\} ~\big|~\left\|  W^{\sigma} \right\|_{\alpha} < \epsilon \right) = 1
	\end{equation}
	for all $c \in \mathbb{R}$, as $\epsilon \to 0$ and $\delta \to 0$.
	
	Based on our previous results \cite{42}, consider a stochastic differential equation
	\begin{displaymath}
		\mathrm{d} X_t = f(t, X_t)\mathrm{d} t + g(t)\mathrm{d}W_t,
	\end{displaymath}
	the Onsager-Machlup functional, which accounts for path deviations caused by drift and perturbations, is given by
	\begin{displaymath}
		\int_{0}^{1} OM(\varphi, \dot{\varphi}) \,{\rm d}s = \int_{0}^{1} \left| {g(s)}^{-1} \left({\dot{\varphi}_s - f(s, \varphi_s)} \right)\right|^2 \,{\rm d}s + \int_{0}^{1} {{\rm div}^{g}_{x} f(s ,\varphi_s)} \,{\rm d}s,
	\end{displaymath}
	where \({\rm div}^{g}_{x} f(s, \varphi_s) := \mathrm{Tr}~\left({g(s)^{-1}}{\nabla f(s ,\varphi_s)}g(s) \right)\) representes a correction term.
	
	For Hamiltonian systems with conserved energy \( H \in C^3_b(\mathbb{R}^n \times \mathbb{R}^n, \mathbb{R}) \), the correction term vanishes. This is a direct consequence of the underlying symplectic structure: the Hamiltonian vector field \( J \nabla H \) is divergence-free, that is, \( \nabla \cdot (J \nabla H) = 0 \).  This structural property, combined with the estimates in \eqref{3.6}, \eqref{3.7}, \eqref{3.9}-\eqref{3.14}, and \eqref{3.16}, implies that the stochastic correction term in the pathwise expansion is identically zero. Consequently, the Onsager-Machlup functional for the stochastic Hamiltonian system reduces to the following explicit form
	\begin{displaymath}
		\begin{aligned}
			\int_{0}^{1} OM(\varphi_q, \varphi_p) \,{\rm d}s = \int_0^1 \left| \sigma_q^{-1}(s) \left( \dot{\varphi}_q - \frac{\partial H}{\partial \varphi_p}(\varphi_q, \varphi_p) \right) \right|^2 \,{\rm d}s + \int_0^1 \left| \sigma_p^{-1}(s) \left( \dot{\varphi}_p + \frac{\partial H}{\partial \varphi_q}(\varphi_q, \varphi_p) \right) \right|^2 \,{\rm d}s.
		\end{aligned}
	\end{displaymath}
\end{proof}

\section{The most probable path in stochastic Hamiltonian systems}
The Onsager-Machlup functional quantifies the relative likelihood of continuous trajectories in the path space of a stochastic dynamical system. Consequently, the minimizer of the Onsager-Machlup functional $ \int_0^1 OM(\varphi_q, \varphi_p) \mathrm{d}t$ corresponds to the most probable path $ \hat{\varphi}(t) := (\hat{\varphi}_q(t), \hat{\varphi}_p(t)) $ in the stochastic Hamiltonian system \eqref{3.1}.

Owing to the specific structure of the Onsager-Machlup functional for such systems, given by
\begin{displaymath}
	\int_0^1 OM(\varphi_q, \varphi_p) \,\mathrm{d}t = \int_0^1 \left| \sigma_q^{-1}(t) \left( \dot{\varphi}_q - \frac{\partial H}{\partial \varphi_p}(\varphi_q, \varphi_p) \right) \right|^2 \mathrm{d}t 
	+ \int_0^1 \left| \sigma_p^{-1}(t) \left( \dot{\varphi}_p + \frac{\partial H}{\partial \varphi_q}(\varphi_q, \varphi_p) \right) \right|^2 \mathrm{d}t,
\end{displaymath}
which is nonnegative and vanishes if and only if $(\varphi_q,\varphi_p)$ satisfies, almost everywhere on $[0,1]$, the deterministic Hamiltonian equations obtained by setting the stochastic perturbation terms in~\eqref{3.1} equal to zero. Thus, the most probable path $\hat{\varphi}(t)$ is precisely the solution of
\begin{equation}\label{4.1}
	\begin{cases}
		\mathrm{d} \varphi_q(t) = \frac{\partial H}{\partial \varphi_p}(\varphi_q(t), \varphi_p(t)) \,{\rm d}t,\\
		\mathrm{d} \varphi_p(t) = -\frac{\partial H}{\partial \varphi_q}(\varphi_q(t), \varphi_p(t)) \,{\rm d}t,
	\end{cases}
\end{equation}
with initial condition $\left( \hat{\varphi}_q(0), \hat{\varphi}_p(0)\right) = \left( q_0, p_0\right) $. This indicates that despite the presence of stochastic perturbations, the system is most likely to evolve along the classical Hamiltonian trajectory.

Therefore, although stochastic perturbations introduce complexity and uncertainty into Hamiltonian systems, by minimizing the Onsager-Machlup functional, we can still reveal the dominant factors of system behavior, namely the evolution along the most probable path. This discovery is of significant importance in both theoretical research and practical applications. To further illustrate this conclusion, we present the following example.
\begin{example}
	Consider a one-dimensional harmonic oscillator with the Hamiltonian $ H(q, p) = \frac{p^2}{2m} + \frac{1}{2} k q^2 $, where $ \frac{p^2}{2m} $ represents the kinetic energy term, and $ V(q) = \frac{1}{2} k q^2 $ is the potential energy term. Here, $ m=1 $ denotes the mass, $k=1$ denotes the spring constant, and $ q $ and $ p $ respectively represent position and momentum. The classical Hamiltonian equations are given by
	\begin{equation*}
		\begin{cases}
			\begin{aligned}
				\mathrm{d}q(t) &= p \,{\rm d}t,\quad &q(0) &= 50,\\
				\mathrm{d}p(t) &= -q \,{\rm d}t,\quad &p(0) &= 0.
			\end{aligned}
		\end{cases}
	\end{equation*}
	Additionally, consider a one-dimensional harmonic oscillator under the influence of white noise perturbations, defined by the following stochastic Hamiltonian system
	\begin{equation*}
		\begin{cases}
			\begin{aligned}
				\mathrm{d}q(t) &= p \,{\rm d}t + (1 + \sin(t)) \,{\rm d}W_q(t), & q(0) &= 50,\\
				\mathrm{d}p(t) &= -q \,{\rm d}t + (1 + 2\cos(3t)) \,{\rm d}W_p(t), & p(0) &= 0,
			\end{aligned}
		\end{cases}
	\end{equation*}
	where $W_q(t)$ and $W_p(t)$ are independent one-dimensional Brownian motions.
\end{example}

In Fig.\ref{F.1}, we observe a clear contrast between the phase space trajectories of the deterministic harmonic oscillator and the stochastically perturbed oscillator. In the absence of perturbations, the deterministic system follows a stable, closed circular orbit, where momentum $ p $ and position $ q $ evolve periodically, reflecting energy conservation. This circular trajectory represents the steady exchange between kinetic and potential energy, maintaining the Hamiltonian.

\begin{figure}
	\fontsize{10}{12}
	\centering
	\includegraphics[width=10cm]{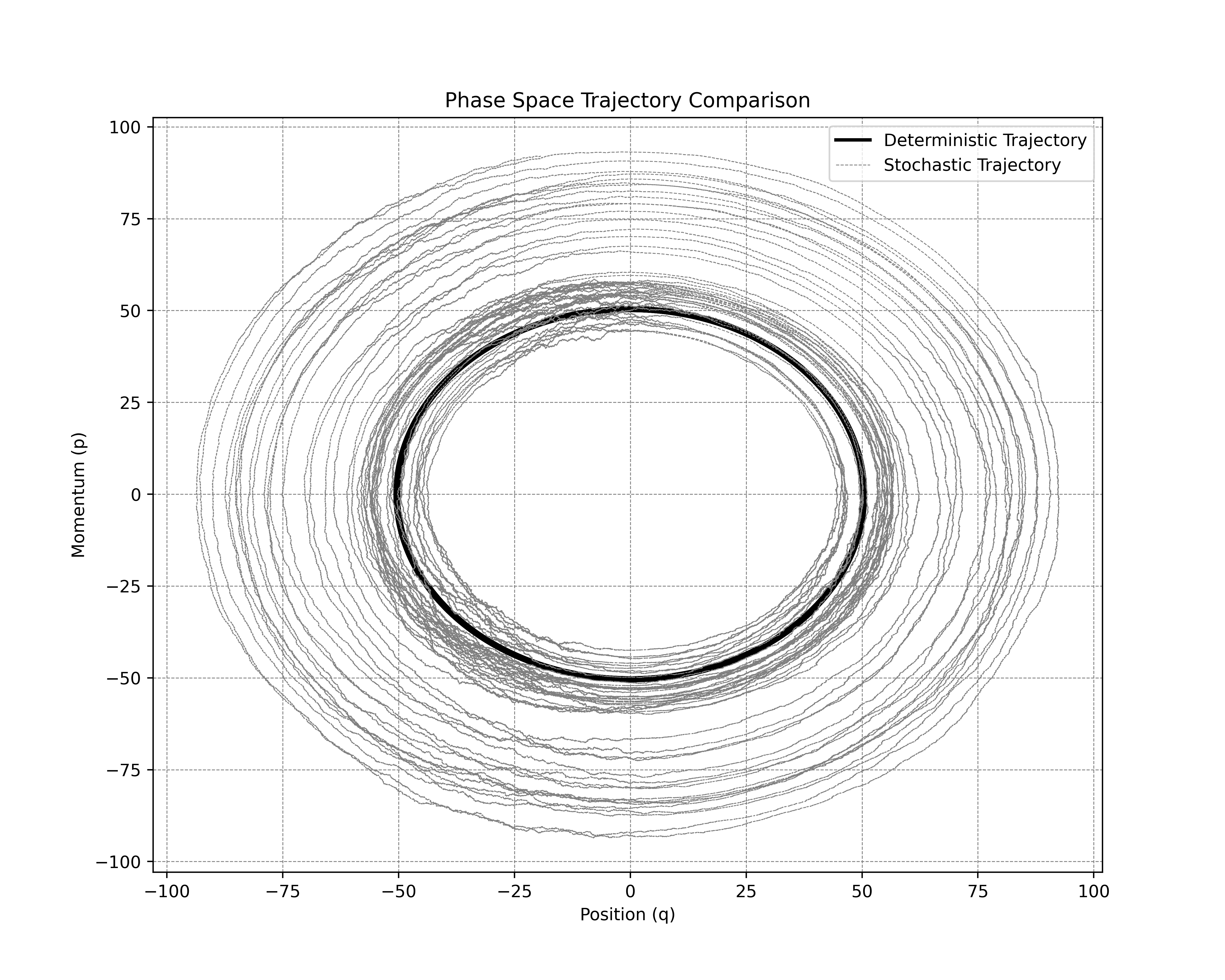}
	\caption{\selectfont Comparison of phase space trajectories between the deterministic (bold solid line) and stochastically perturbed (thin dashed lines) harmonic oscillators. The system is modeled with the following parameters: mass \(m = 1\), spring constant \(k = 1\), time step \(\Delta t = 0.0001\), and total simulation time \(T = 300\). The deterministic trajectory follows a stable circular orbit in phase space, indicating energy conservation and periodic motion. In contrast, the stochastic system introduces white noise perturbations using the Euler-Maruyama method, leading to diffusive and irregular trajectories that deviate from the original orbit over time.}
	\label{F.1}
\end{figure}

When white noise perturbations are introduced, however, the trajectory deviates from this ideal path, displaying diffusion and irregularity. As shown in Fig.\ref{F.1}, the stochastic oscillator's trajectory gradually drifts away from the stable orbit, with random fluctuations disrupting the original pattern. This diffusion occurs because the noise introduces energy fluctuations, causing deviations in $ p $ and $ q $ that prevent strict adherence to the classical oscillator's path. While some periodic behavior remains, the perturbations induce instability, leading to a trajectory that expands outward in phase space over time.

Nevertheless, we also observe that the stochastic oscillator's phase space trajectories remain densely concentrated near the deterministic oscillator's stable circular orbit. This aligns with our theoretical conclusions: the most probable path of the stochastically perturbed system stays close to the stable deterministic trajectory. This reinforces the understanding that, despite stochastic perturbations, the system tends to evolve near the deterministic solution.

Fig.\ref{F.2} provides further insight by displaying the probability distribution of the Hamiltonian $ H(q, p) $ for the stochastically perturbed oscillator. Notably, the Hamiltonian distribution $ H_S $ peaks near $ H_S = 1249.6399 $, closely matching the deterministic Hamiltonian value $ H_D = 1250 $. This behavior supports our theoretical conclusion that the most probable Hamiltonian value for the stochastically perturbed system coincides with that of the deterministic system. For most of the time, the system's Hamiltonian remains concentrated near the deterministic value, with large energy fluctuations being rare. This indicates that, despite stochastic perturbations, the system's overall evolution largely adheres to the conservation properties of the Hamiltonian system, with only local deviations due to noise.

\begin{figure}[htbp]
	\centering
	\includegraphics[width=11cm]{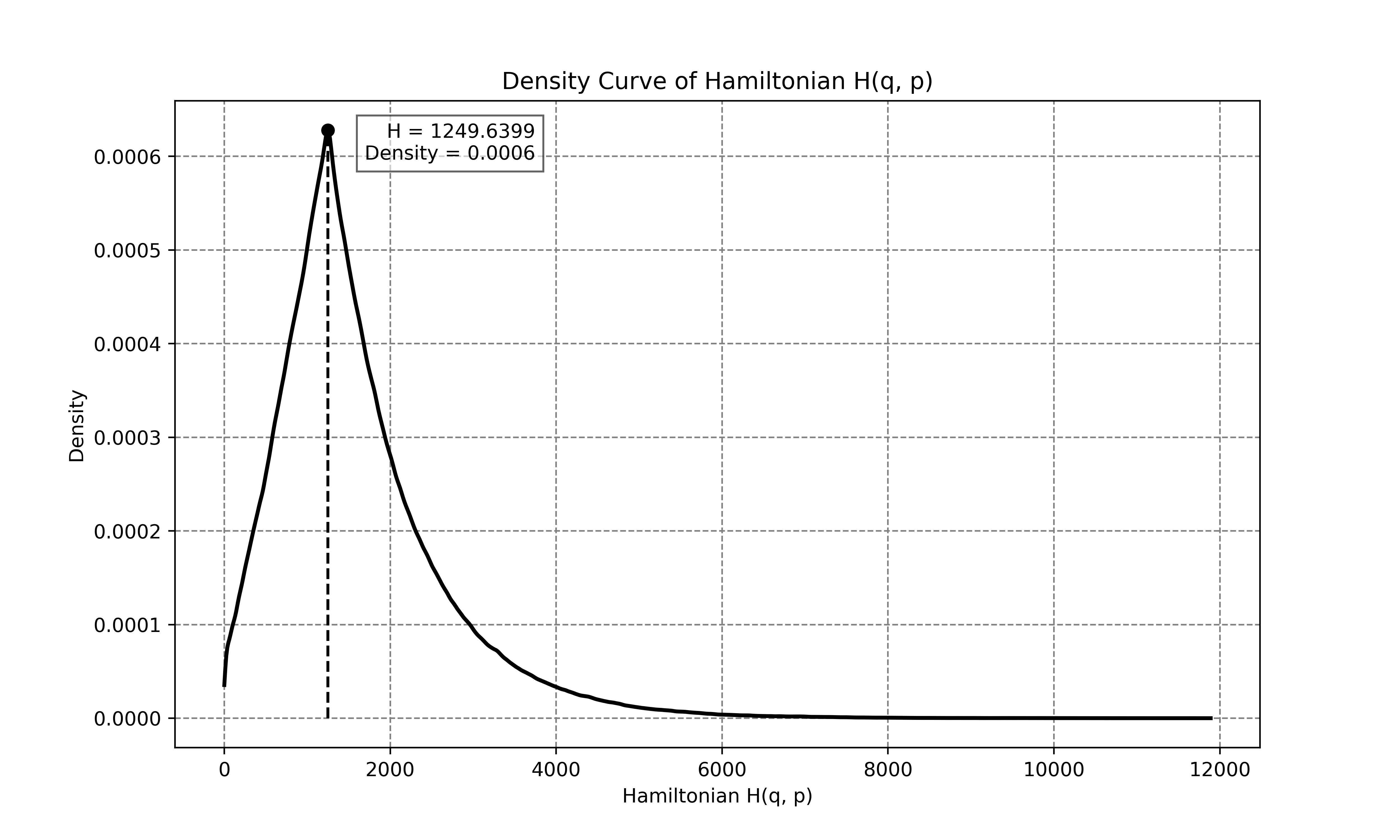}
	\caption{The distribution curve of the Hamiltonian $ H(q, p) = \frac{p^2}{2m} + \frac{1}{2}kq^2 $. The system is modeled with the following parameters: mass $ m = 1 $, spring constant $ k = 1 $, time step $ \Delta t = 0.001 $, total simulation time $ T = 300 $, and 5000 simulation runs. The maximum value of the Hamiltonian $ H $ is observed at 1249.6399, with a corresponding density of 0.0006. This result illustrates the distribution of the Hamiltonian in the presence of stochastic perturbations, showing that the system's energy tends to concentrate around specific values in most cases.}
	\label{F.2}
\end{figure}

In summary, Figures \ref{F.1} and \ref{F.2} jointly illustrate the dynamics of both the classical Hamiltonian system and its stochastically perturbed counterpart, comparing their phase-space trajectories and energy evolution. The numerical results demonstrate that, despite the presence of time-dependent additive noise, the most probable trajectory remains exactly the unperturbed Hamiltonian flow. More importantly, this observation reveals a structural stability of Hamiltonian systems in random environments: the conservation properties endowed by symplectic geometry exhibit strong robustness under stochastic perturbations. Furthermore, the noise intensity does not alter the most probable path itself, but only modifies its likelihood,i.e., the probability weight assigned to that trajectory by the Onsager-Machlup functional.

\section{Large Deviation Principle and Rate Function for Stochastic Hamiltonian Systems}

In this section, we derive the large deviation principle for stochastic Hamiltonian systems, combining the Onsager-Machlup functional and Freidlin-Wentzell theory \cite{46}. Large deviation principle describes the probability behavior of a system's trajectory deviating from the most probable path, with the rate function quantifying the decay rate of this probability.

Since our objective is to study the behavior of the trajectories in the stochastic Hamiltonian system as the noise term approaches zero, we simplify equation \eqref{3.1} into the following form for convenience,
\begin{equation}\label{5.1}
	\begin{cases}
		\mathrm{d}q(t) = \frac{\partial H}{\partial p}(q, p)\,\mathrm{d}t + \gamma \sigma_q(t)\,\mathrm{d}W_q(t), \\
		\mathrm{d}p(t) = -\frac{\partial H}{\partial q}(q, p)\,\mathrm{d}t + \gamma \sigma_p(t)\,\mathrm{d}W_p(t),
	\end{cases}
\end{equation}
with initial condition $\left( q(0) , p(0)\right)  = \left( q_0 , p_0\right)$. Here \( q(t) \) and \( p(t) \) represent the generalized position and momentum variables, respectively, and \( H(q, p) \) is the Hamiltonian function describing the system's energy, typically composed of kinetic and potential energy terms. \( W_q(t) \) and \( W_p(t) \) are independent Wiener processes, and the matrices \( \sigma_q(t) \) and \( \sigma_p(t) \) denote the diffusion coefficients for the stochastic components. The parameter \( \gamma \) represents the noise intensity. Our primary focus is on analyzing the statistical behavior of this system as \( \gamma \to 0 \). Similar to Section 3, our result holds for any finite interval \( t \in [0, T] \). For simplicity of notation, we will present the following theorem on the interval \([0, 1]\).

\begin{theorem} \label{T5.1} 
	Assuming that Conditions $(C1)$ and $(C2)$ hold, let $ X^\gamma(t) := (q^\gamma(t), p^\gamma(t)) $ be the solution to stochastic Hamiltonian system \eqref{5.1} with initial condition $x_0 := X^\gamma(0) = \left( q_0, p_0\right) $. As $ \gamma \to 0 $, the most probable path is given by the deterministic Hamiltonian equations \eqref{4.1}. Furthermore, fix $\alpha\in(0,\tfrac14)$. The family of probability measures on $C^\alpha([0,1];\mathbb{R}^{n} \times \mathbb{R}^{n})$ induced by $X^\gamma(\cdot)$ and equipped with the Hölder topology, satisfies a large deviation principle with speed $\gamma^2$ and rate function
	\begin{equation*}
		\begin{aligned}
			J(\varphi)=
			\begin{cases}
				\frac{1}{2}\left( \int_0^1 \left| \sigma_q^{-1}(t) \left( \dot{\varphi}_q - \frac{\partial H}{\partial \varphi_p}(\varphi_q, \varphi_p) \right) \right|^2 \,{\rm d}t + \int_0^1 \left| \sigma_p^{-1}(t) \left( \dot{\varphi}_p + \frac{\partial H}{\partial \varphi_q}(\varphi_q, \varphi_p) \right) \right|^2 \,{\rm d}t \right),\\  & \hspace{-2cm}\text{if } \varphi - x_0 \in \mathbb{H}^{1}_0;\\
				+\infty, & \hspace{-2cm} \text{otherwise},
			\end{cases}
		\end{aligned}
	\end{equation*}
	with \( \sigma_q^{-1}(t) \) and \( \sigma_p^{-1}(t) \) being the inverses of the diffusion matrices \( \sigma_q(t) \) and \( \sigma_p(t) \), respectively. More precisely, for every Borel set $\mathbb{A} \subset C^\alpha([0,1];\mathbb{R}^{n} \times \mathbb{R}^{n})$,
	\[
	-\inf_{\varphi \in \mathbb{A}^\circ} J(\varphi) \le
	\liminf_{\gamma \to 0}\gamma^2 \log \mathbb{P}(X^\gamma(\cdot) \in \mathbb{A})\le
	\limsup_{\gamma \to 0}\gamma^2 \log \mathbb{P}(X^\gamma(\cdot) \in \mathbb{A})\le
	-\inf_{\varphi \in \bar{\mathbb{A}} } J(\varphi).
	\]
\end{theorem}

\begin{proof}
	Since the sample paths are almost surely $\alpha$-Hölder continuous for every $\alpha<\frac12$, we may regard the law of $X^\gamma(\cdot)$ as a probability measure on $C^\alpha([0,1];\mathbb{R}^n\times\mathbb{R}^n)$). Moreover, the Sobolev embedding $\mathbb{H}_0^1 \hookrightarrow C^\alpha$ is continuous for any $\alpha < \frac12$. We first establish an estimate for the probability that the solution trajectories $X^\gamma(t)$ of the stochastic Hamiltonian system remain within an arbitrary tubular neighborhood. We then prove the lower and upper bounds in Definition~\ref{D2.8}, thereby completing the proof of the large deviation theorem.
	
	Equation $\eqref{5.1}$ merely concretizes the small random noise in equation $\eqref{3.1}$ as being of order $\gamma$, where $\gamma$ is a small parameter. Therefore, we employ a definition and method similar to those used in the proof of Theorem $\ref{T3.1}$. Let the reference path be given by $\varphi(t) = (\varphi_q(t), \varphi_p(t))$, where $\varphi(t)$ is a definite continuous path, $ \varphi(0) = x_0$ and $ \varphi - x_0 \in \mathbb{H}^{1}_0 $. We define the perturbed solution, denoted as $(y_q(t), y_p(t))$,
	\begin{displaymath}
		\begin{cases}
			y_q(t) = \varphi_q(t) + \gamma \int_0^t \sigma_q(s)\, {\rm d}W_q(s),\\
			y_p(t) = \varphi_p(t) + \gamma \int_0^t \sigma_p(s)\, {\rm d}W_p(s).
		\end{cases}
	\end{displaymath}
	To simplify the notation in the proof, we introduce the term $ W^{\sigma}(t) := \left( W^{\sigma}_q(t), W^{\sigma}_p(t) \right) $, which encapsulates the stochastic perturbation in the system
	\begin{displaymath}
		W^{\sigma}_q(t) := \int_0^t \sigma_q(s) \,{\rm d}W_q(s), \quad W^{\sigma}_p(t) := \int_0^t \sigma_p(s) \,{\rm d}W_p(s).
	\end{displaymath}
	
	Then we introduce a new probability measure \(\tilde{\mathbb{P}}\), under which the transformed Brownian motions are given by
	\begin{displaymath}
		\begin{aligned}
			\tilde{W}_q(t) &= W_q(t) - \frac{1}{\gamma} \int_0^t \sigma_q^{-1}(s) \left( \frac{\partial H}{\partial y_p}(y_q, y_p) - \dot{\varphi}_q(s) \right)\, {\rm d}s,\\
			\tilde{W}_p(t) &= W_p(t) - \frac{1}{\gamma} \int_0^t \sigma_p^{-1}(s) \left( -\frac{\partial H}{\partial y_q}(y_q, y_p) - \dot{\varphi}_p(s) \right)\, {\rm d}s.
		\end{aligned}
	\end{displaymath}
	
	The Radon-Nikodym derivative $\mathcal{R} := \frac{d\tilde{\mathbb{P}}}{d\mathbb{P}}$, representing the change of measure from $\mathbb{P}$ to $\tilde{\mathbb{P}}$, is given by the exponential martingale associated with the removed drift terms. Thus,
	\begin{displaymath}
		\begin{aligned}
			\mathcal{R} = \exp &\left\lbrace \frac{1}{\gamma} \left( \int_0^1 \left\langle \sigma_q^{-1}(s) \left( \frac{\partial H}{\partial y_p}(y_q, y_p) - \dot{\varphi}_q(s) \right), \,{\rm d}W_q(s) \right\rangle\right. \right.\\
			&\qquad \left. - \int_0^1 \left\langle \sigma_p^{-1}(s) \left( \frac{\partial H}{\partial y_q}(y_q, y_p) + \dot{\varphi}_p(s) \right), \,{\rm d}W_p(s) \right\rangle \right)\\ 
			& - \frac{1}{2 \gamma^2} \left( \int_0^1 \left| \sigma_q^{-1}(s) \left( \frac{\partial H}{\partial y_p}(y_q, y_p) - \dot{\varphi}_q(s) \right) \right|^2\, {\rm d}s\right. \\
			&\qquad \left. \left. + \int_0^1 \left| \sigma_p^{-1}(s) \left( \frac{\partial H}{\partial y_q}(y_q, y_p) + \dot{\varphi}_p(s) \right) \right|^2\, {\rm d}s \right) \right\rbrace.
		\end{aligned}
	\end{displaymath}
	
	Define
	\[
	\mathbb{K}(\varphi, \epsilon) = \{ x - x_0 \in \mathbb{H}^1_0 \mid \varphi - x_0 \in \mathbb{H}^1_0, \|x - \varphi\|_{\alpha} \leq \epsilon, \epsilon > 0 \}.
	\]
	Utilizing Girsanov's Theorem, we aim to compute the probability that the trajectory of the solution $ X^{\gamma}(t) $ of the stochastic Hamiltonian system remains in close proximity to a reference path $ \varphi(t) $ when the noise intensity $ \epsilon $ is minimal. Then
	\begin{equation}\label{52}
		\begin{aligned}
			&\quad \mathbb{P}(X^{\gamma}(\cdot) \in \mathbb{K}(\varphi, \epsilon) ) =  \mathbb{P}\left(\left\|  (q^{\gamma}, p^{\gamma}) - (\varphi_q, \varphi_p) \right\|_{\alpha} \leq \epsilon \right)\\
			&= \tilde{\mathbb{P}}\left(\left\| (Y_q, Y_p) -(\varphi_q, \varphi_p) \right\|_{\alpha} \leq \epsilon \right)
			= \mathbb{E} \left( \mathcal{R} \mathbb{I}_{ \left\| W^{\sigma} \right\|_{\alpha} \leq \frac{\epsilon}{\gamma}} \right) \\
			&= \mathbb{E} \left( \mathcal{R} ~\big|~ { \left\| W^{\sigma} \right\|_{\alpha} \leq \frac{\epsilon}{\gamma}} \right)  \times \mathbb{P}( \left\| W^{\sigma} \right\|_{\alpha} \leq \frac{\epsilon}{\gamma} ) \\ 
			&= \exp\left\lbrace -\frac{1}{2\gamma^2} \left(  \int_0^1 \left| \sigma_q^{-1}(s) \left( \dot{\varphi}_q - \frac{\partial H}{\partial \varphi_p}(\varphi_q, \varphi_p) \right) \right|^2\, {\rm d}s + \int_0^1 \left| \sigma_p^{-1}(s) \left( \dot{\varphi}_p + \frac{\partial H}{\partial \varphi_q}(\varphi_q, \varphi_p) \right) \right|^2\, {\rm d}s \right) \right\rbrace\\
			&\quad \times \mathbb{E} \left( \exp\left\lbrace \frac{1}{\gamma^2} \sum_{i=1}^{6} B_i \right\rbrace  ~\big|~ { \left\| W^{\sigma} \right\|_{\alpha} \leq \frac{\epsilon}{\gamma}} \right) \times \mathbb{P}( \left\| W^{\sigma} \right\|_{\alpha} \leq \frac{\epsilon}{\gamma} ),
		\end{aligned}
	\end{equation}
	where $B_i$ represents the deviations in the path arising from drift and perturbations, it exhibits stochastic properties. This is elaborated upon in the following detailed expressions:
	\begin{displaymath}
		\begin{aligned}
			B_1 &=  \gamma \int_0^1 \left\langle \sigma_q^{-1}(s) \frac{\partial H}{\partial y_p}(y_q, y_p) , \,{\rm d}W_q(s) \right\rangle -  \gamma \int_0^1 \left\langle \sigma_p^{-1}(s)  \frac{\partial H}{\partial y_q}(y_q, y_p), \,{\rm d}W_p(s) \right\rangle,\\
			B_2 &= - \gamma \int_0^1 \left\langle \sigma_q^{-1}(s) \dot{\varphi}_q(s), \,{\rm d}W_q(s) \right\rangle - \gamma \int_0^1 \left\langle \sigma_p^{-1}(s) \dot{\varphi}_p(s), \,{\rm d}W_p(s) \right\rangle,\\
			B_3 &= \frac{1}{2} \int_0^1 \left| \sigma_q^{-1}(s) \frac{\partial H}{\partial \varphi_p}(\varphi_q, \varphi_p)  \right|^2 \,{\rm d}s - \frac{1}{2} \int_0^1 \left| \sigma_q^{-1}(s)  \frac{\partial H}{\partial y_p}(y_q, y_p) \right|^2 \,{\rm d}s,\\
			B_4 &= \frac{1}{2} \int_0^1 \left| \sigma_p^{-1}(s) \frac{\partial H}{\partial \varphi_q}(\varphi_q, \varphi_p)  \right|^2 \,{\rm d}s - \frac{1}{2} \int_0^1 \left| \sigma_p^{-1}(s)  \frac{\partial H}{\partial y_q}(y_q, y_p) \right|^2 \,{\rm d}s,\\
			B_5 &= \int_0^1 \left\langle \sigma_q^{-1}(s) \left( \frac{\partial H}{\partial y_p}(y_q, y_p) - \frac{\partial H}{\partial \varphi_p}(\varphi_q, \varphi_p) \right), \sigma_q^{-1}(s) \dot{\varphi}_q(s) \right\rangle \,{\rm d}s,\\
			B_6 &= - \int_0^1 \left\langle \sigma_p^{-1}(s) \left( \frac{\partial H}{\partial y_q}(y_q, y_p) - \frac{\partial H}{\partial \varphi_q}(\varphi_q, \varphi_p) \right), \sigma_p^{-1}(s) \dot{\varphi}_p(s) \right\rangle \,{\rm d}s.
		\end{aligned}
	\end{displaymath}

	We decompose the system's probability into a deterministic part
	\begin{displaymath}
		\begin{aligned}
			\exp \left\lbrace -\frac{1}{2\gamma^2}\left( \int_0^1 \left| \sigma_q^{-1}(s) \left( \dot{\varphi}_q - \frac{\partial H}{\partial \varphi_p}(\varphi_q, \varphi_p) \right) \right|^2\, {\rm d}s + \int_0^1 \left| \sigma_p^{-1}(s) \left( \dot{\varphi}_p + \frac{\partial H}{\partial \varphi_q}(\varphi_q, \varphi_p) \right) \right|^2\, {\rm d}s \right) \right\rbrace
		\end{aligned}
	\end{displaymath}
	a correction part
	\begin{displaymath}
		\mathbb{E}\left( \exp\left\lbrace \frac{1}{\gamma^2} \sum_{i=1}^{6} B_i ~\big|~ {\left\|W^{\sigma}\right\|_{\alpha} \leq \frac{\epsilon}{\gamma}} \right\rbrace \right),
	\end{displaymath}
	and small ball probabilities for Brownian motion
	\begin{displaymath}
		\mathbb{P}\left(  \left\| W^{\sigma} \right\|_{\alpha} \leq \frac{\epsilon}{\gamma}\right).
	\end{displaymath}
	Specifically, the deterministic component represents the behavior along the deterministic trajectory, while the correction component accounts for deviations induced by stochastic perturbations. The small ball probabilities for Brownian motion represent fixed values that are intrinsically linked to the diffusion coefficients $\sigma_q$, $\sigma_p$, and the duration of time, yet they are independent of the variable $\varphi$. Furthermore, for convenience, we introduce the notation \( U(\varphi) \) defined as
	\[
	U(\varphi) := \left| \sigma_{q}^{-1}(t) \left( \dot{\varphi}_{q} - \frac{\partial H}{\partial \varphi_{p}}(\varphi_{q}, \varphi_{p}) \right) \right|^{2} + \left| \sigma_{p}^{-1}(t) \left( \dot{\varphi}_{p} + \frac{\partial H}{\partial \varphi_{q}}(\varphi_{q}, \varphi_{p}) \right) \right|^{2}.
	\]
	
	Large deviation theory focuses on large-scale deviations, and in this context, we are particularly concerned with the \(\frac{1}{\gamma^2}\) scale. Given the boundedness of \(\sigma_{q}^{-1}(s)\) and \(\sigma_{p}^{-1}(s)\), along with the fact that \(H \in C^3_b(\mathbb{R}^n \times \mathbb{R}^n, \mathbb{R})\), we can infer that
	\begin{equation}\label{5.3}
		\limsup\limits_{\gamma \to 0} \mathbb{E}\left({\rm exp}\left\{ cB_{1} \right\} ~\big|~ { \left\| W^{\sigma} \right\|_{\alpha} \leq \frac{\epsilon}{\gamma}} \right) = \exp \left\lbrace O(\epsilon)  \right\rbrace,
	\end{equation}
	and
	\begin{equation}\label{5.4}
		\limsup\limits_{\gamma \to 0} \mathbb{E}\left({\rm exp}\left\{ cB_{2} \right\} ~\big|~ { \left\| W^{\sigma} \right\|_{\alpha} \leq \frac{\epsilon}{\gamma}} \right) = \exp \left\lbrace O(\epsilon)  \right\rbrace
	\end{equation}
	for all $c \in \mathbb{R}$.
	
	For the third term $B_3$, 
	\begin{displaymath}
		\begin{aligned}
			B_3 &= \frac{1}{2} \int_0^1 \left| \sigma_q^{-1}(s) \frac{\partial H}{\partial \varphi_p}(\varphi_q, \varphi_p)  \right|^2 \,{\rm d}s - \frac{1}{2} \int_0^1 \left| \sigma_q^{-1}(s)  \frac{\partial H}{\partial y_p}(y_q, y_p) \right| ^2 \,{\rm d}s
			\\ &\leq \frac{1}{2} \int_{0}^{1} {\left| \left\| \sigma_q^{-1}(s)\right\| \right|^2 \left|  \frac{\partial H}{\partial \varphi_p}(\varphi_q, \varphi_p) - \frac{\partial H}{\partial y_p}(y_q, y_p) \right|^2} \,{\rm d}s\\
			& \quad + \int_{0}^{1} {\left| \left\| \sigma_q^{-1}(s)\right\| \right|^2 \left|\frac{\partial H}{\partial \varphi_p}(\varphi_q, \varphi_p) - \frac{\partial H}{\partial y_p}(y_q, y_p) \right| \left| \frac{\partial H}{\partial y_p}(y_q, y_p) \right|} \,{\rm d}s.
		\end{aligned}
	\end{displaymath}
	Using that $\frac{\partial H}{\partial p}$ is Lipschitz continuous, we have
	\begin{equation}\label{5.5}
		\begin{aligned}
			&\quad \left| \frac{\partial H}{\partial y_p}(y_q, y_p)  -  \frac{\partial H}{\partial \varphi_p}(\varphi_q, \varphi_p)\right|\\
			&= \left| \frac{\partial H}{\partial \left( \varphi_p + \gamma W^{\sigma}_p \right)}((\varphi_q + \gamma W^{\sigma}_q), (\varphi_p + \gamma W^{\sigma}_p))-  \frac{\partial H}{\partial \varphi_p}(\varphi_q, \varphi_p)\right|\\
			&\leq L \gamma \left\| W^{\sigma} \right\|_{\alpha}.
		\end{aligned}
	\end{equation}
	Inequality $\eqref{5.5}$ and the boundedness of $\frac{\partial H}{\partial y_p}(y_q, y_p)$ and $\sigma_q^{-1}(t)$ imply that
	\begin{equation}\label{5.6}
		\limsup\limits_{\gamma \to 0} \mathbb{E}\left({\rm exp}\left\{ cB_3 \right\} ~\big|~ { \left\| W^{\sigma} \right\|_{\alpha} \leq \frac{\epsilon}{\gamma}} \right) = \exp \left\lbrace O(\epsilon)  \right\rbrace
	\end{equation}
	for all $c \in \mathbb{R}$.
	
	For the fourth term $B_4$, employing the same proof technique as for the third term $B_3$, we have
	\begin{equation}\label{5.7}
		\limsup\limits_{\gamma \to 0} \mathbb{E}\left({\rm exp}\left\{ cB_4 \right\} ~\big|~ { \left\| W^{\sigma} \right\|_{\alpha} \leq \frac{\epsilon}{\gamma}} \right) = \exp \left\lbrace O(\epsilon)  \right\rbrace
	\end{equation}
	for all $c \in \mathbb{R}$.
	
	For the fifth term $B_5$, applying inequality $\eqref{5.5}$ and the boundedness of $\dot{\varphi}_\theta(t)$ and $\sigma_\theta^{-1}(t)$, we have
	\begin{displaymath}
		\begin{aligned}
			B_5 &= \int_0^1 \left\langle \sigma_q^{-1}(s) \left( \frac{\partial H}{\partial y_p}(y_q, y_p) - \frac{\partial H}{\partial \varphi_p}(\varphi_q, \varphi_p) \right), \sigma_q^{-1}(s) \dot{\varphi}_q(s) \right\rangle \,{\rm d}s\\
			& \leq CL\gamma \left\| W^{\sigma} \right\|_{\alpha}.
		\end{aligned}
	\end{displaymath}
	Thus,
	\begin{equation}\label{5.8}
		\limsup\limits_{\gamma \to 0} \mathbb{E}\left({\rm exp}\left\{ cB_5 \right\} ~\big|~ { \left\| W^{\sigma} \right\|_{\alpha} \leq \frac{\epsilon}{\gamma}} \right) = \exp \left\lbrace O(\epsilon)  \right\rbrace
	\end{equation}
	for all $c \in \mathbb{R}$.
	
	For the sixth term $B_6$, employing the same proof technique as for the fifth term $B_5$, we have
	\begin{equation}\label{5.9}
		\limsup\limits_{\gamma \to 0} \mathbb{E}\left({\rm exp}\left\{ cB_6 \right\} ~\big|~ { \left\| W^{\sigma} \right\|_{\alpha} \leq \frac{\epsilon}{\gamma}} \right) = \exp \left\lbrace O(\epsilon)  \right\rbrace
	\end{equation}
	for all $c \in \mathbb{R}$.

	Based on the outcomes derived from equations $\eqref{52}$-$\eqref{5.4}$ and $\eqref{5.6}$-$\eqref{5.9}$, we can ascertain the explicit form of the rate function by utilizing both the upper and lower bounds of the probability estimate. Here, we solely focus on the scenario where $ \varphi(0) = x_0 $ and $\varphi - x_0 \in \mathbb{H}^{1}_0$, and for the remaining scenarios, we directly assign $J = \infty$.
	
	On one hand, we estimate the upper bound. Let $\mathbb{F}$ be a closed subset of
	$C^\alpha([0,1];\mathbb{R}^n \times \mathbb{R}^n)$. By exponential tightness of $\{X^\gamma\}_{\gamma>0}$
	in $C^\alpha([0,1];\mathbb{R}^n \times \mathbb{R}^n)$ (see~[\cite{36}]), for every $K>0$ there exists a compact set
	$\mathbb{K}_K\subset C^\alpha([0,1];\mathbb{R}^n \times \mathbb{R}^n)$ such that
	\begin{equation}\label{5.10}
		\limsup_{\gamma\to0}\gamma^2\log\mathbb{P}\left( X^\gamma(\cdot)\notin \mathbb{K}_K\right) \le -K .
	\end{equation}
	Hence,
	\begin{equation}\label{5.11}
		\mathbb{P}\big(X^\gamma(\cdot)\in\mathbb{F}\big)
		\le \mathbb{P}\big(X^\gamma(\cdot)\in\mathbb{F}\cap \mathbb{K}_K\big)
		+\mathbb{P}\big(X^\gamma(\cdot)\notin \mathbb{K}_K\big).
	\end{equation}
	
	Fix $K>0$ and $\epsilon>0$. Since $\mathbb{F}\cap \mathbb{K}_K$ is compact, there exist
	$\{\varphi_j\}_{j=1}^{N(\epsilon,K)}\subset \mathbb{F}\cap \mathbb{K}_K$ such that
	\[
	\mathbb{F}\cap \mathbb{K}_K \subset \bigcup_{j=1}^{N(\epsilon,K)} \mathbb{K}(\varphi_j,\epsilon).
	\]
	By the union bound,
	\[
	\mathbb{P}\big(X^\gamma(\cdot)\in\mathbb{F}\cap \mathbb{K}_K\big)
	\le \sum_{j=1}^{N(\epsilon,K)} \mathbb{P}\big(X^\gamma(\cdot)\in \mathbb{K}(\varphi_j,\epsilon)\big).
	\]
	Taking logarithms and multiplying by $\gamma^2$, we obtain
	\[
	\gamma^2\log \mathbb{P}\big(X^\gamma(\cdot)\in\mathbb{F}\cap \mathbb{K}_K\big)
	\le \gamma^2\log N(\epsilon,K)
	+\max_{1\le j\le N(\epsilon,K)} \gamma^2\log \mathbb{P}\big(X^\gamma(\cdot)\in \mathbb{K}(\varphi_j,\epsilon)\big).
	\]
	Since $N(\epsilon,K)$ does not depend on $\gamma$, we have $\gamma^2\log N(\epsilon,K)\to0$ as $\gamma\to0$.
	Using the local estimate for $\mathbb{P}(X^\gamma(\cdot)\in \mathbb{K}(\varphi_j,\epsilon))$ derived above, we get
	\[
	\limsup_{\gamma\to0}\gamma^2\log \mathbb{P}\big(X^\gamma(\cdot)\in\mathbb{F}\cap \mathbb{K}_K\big)
	\le -\inf_{\varphi\in\mathbb{F}\cap \mathbb{K}_K}\frac12\int_0^1 U(\varphi)\,\mathrm{d}s + O(\epsilon).
	\]
	Letting $\epsilon \to 0$ yields
	\[
	\limsup_{\gamma\to0}\gamma^2\log \mathbb{P}\big(X^\gamma(\cdot)\in\mathbb{F}\cap \mathbb{K}_k\big)
	\le -\inf_{\varphi\in\mathbb{F}\cap \mathbb{K}_k}\frac12\int_0^1 U(\varphi)\,\mathrm{d}s .
	\]
	
	For any nonnegative numbers $a_\gamma,b_\gamma$,
	\[
	\log(a_\gamma+b_\gamma)\le \log 2+\max\{\log a_\gamma,\log b_\gamma\}.
	\]
	Consequently, 
	\[
	\limsup_{\gamma\to0}\gamma^2\log(a_\gamma+b_\gamma)
	\le \max\left\{\limsup_{\gamma\to0}\gamma^2\log a_\gamma,\,
	\limsup_{\gamma\to0}\gamma^2\log b_\gamma\right\}.
	\]
	Combining \eqref{5.11} with the elementary bound, we obtain
	\[
	\limsup_{\gamma\to0}\gamma^2\log \mathbb{P}\big(X^\gamma(\cdot)\in\mathbb{F}\big)
	\le \max\left\{
	\limsup_{\gamma\to0}\gamma^2\log \mathbb{P}\big(X^\gamma(\cdot)\in\mathbb{F}\cap \mathbb{K}_K\big),\,
	\limsup_{\gamma\to0}\gamma^2\log \mathbb{P}\big(X^\gamma(\cdot)\notin \mathbb{K}_K\big)
	\right\}.
	\]
	Using \eqref{5.10} and the estimate on $\mathbb{F}\cap \mathbb{K}_K$, we conclude that for every $K>0$,
	\[
	\limsup_{\gamma\to0}\gamma^2\log \mathbb{P}\big(X^\gamma(\cdot)\in\mathbb{F}\big)
	\le \max\left\{
	-\inf_{\varphi\in\mathbb{F}\cap \mathbb{K}_K}\frac12\int_0^1 U(\varphi)\,\mathrm{d}s,\,-K
	\right\}.
	\]
	Since $\inf_{\mathbb{F}\cap \mathbb{K}_K}(\cdot)\ge \inf_{\mathbb{F}}(\cdot)$, letting $K \to \infty$ yields
	\begin{equation}\label{5.12}
		\limsup_{\gamma\to0}\gamma^2\log \mathbb{P}\big(X^\gamma(\cdot)\in\mathbb{F}\big)
		\le -\inf_{\varphi\in\mathbb{F}}\frac12\int_0^1 U(\varphi)\,\mathrm{d}s.
	\end{equation}

	On the other hand, we estimate the lower bound. Let $\mathbb{G}$ be an arbitrary open set in
	$C^\alpha([0,1];\mathbb{R}^n\times\mathbb{R}^n)$. Fix $\varphi\in\mathbb{G}$.
	Since $\mathbb{G}$ is open in the $C^\alpha$-topology, there exists $\epsilon_0>0$ such that
	$\mathbb{K}(\varphi,\epsilon)\subset \mathbb{G}$ for all $\epsilon\in(0,\epsilon_0)$. Hence, for any $\epsilon\in(0,\epsilon_0)$,
	\[
	\mathbb{P}\big(X^\gamma(\cdot)\in\mathbb{G}\big)\ge
	\mathbb{P}\big(X^\gamma(\cdot)\in\mathbb{K}(\varphi,\epsilon)\big).
	\]
	Taking $\liminf\limits_{\gamma\to0}\gamma^2\log(\cdot)$ and using the local lower bound proved above, we obtain
	\[
	\liminf_{\gamma\to0}\gamma^2\log\mathbb{P}\big(X^\gamma(\cdot)\in\mathbb{G}\big)
	\ge -\frac12\int_0^1 U(\varphi)\,ds - O(\epsilon),
	\]
	where $O(\epsilon) \to 0$ as $\epsilon \to 0$. Since $\varphi\in\mathbb{G}$ is arbitrary, letting $\epsilon\to0$ yields
	\begin{equation}\label{5.13}
		\liminf_{\gamma\to0}\gamma^2\log\mathbb{P}\big(X^\gamma(\cdot)\in\mathbb{G}\big)
		\ge -\inf_{\varphi\in\mathbb{G}}\frac12\int_0^1 U(\varphi)\,\mathrm{d}s.
	\end{equation}

	By combining the upper bound inequality \eqref{5.12} and the lower bound inequality \eqref{5.13}, we can derive the rate function for the stochastic Hamiltonian system
	\begin{displaymath}
		\begin{aligned}
			J(\varphi) &= \frac{1}{2}\left( \int_0^1 \left| \sigma_q^{-1}(s) \left( \dot{\varphi}_q - \frac{\partial H}{\partial \varphi_p}(\varphi_q, \varphi_p) \right) \right|^2\, {\rm d}s + \int_0^1 \left| \sigma_p^{-1}(s) \left( \dot{\varphi}_p + \frac{\partial H}{\partial \varphi_q}(\varphi_q, \varphi_p) \right) \right|^2\, {\rm d}s \right).
		\end{aligned}
	\end{displaymath}
\end{proof}

\section{Preservation of Invariant Tori in Nearly Integrable Stochastic Hamiltonian Systems}
This section is devoted to proving Theorem $\ref{T6.1}$. Working within the nearly integrable regime, we combine the Onsager-Machlup functional, the characterization of the most probable path, and the Freidlin-Wentzell large deviation principle established in Sections 3-5 with the deterministic KAM mechanism. This synthesis yields the persistence of invariant tori under small stochastic perturbations in the most probable paths sense. To avoid an overload of notation, we adopt throughout the conventions, function spaces, and assumptions introduced at the beginning of this section. The results stated below hold on any finite time interval $[0,T]$. For simplicity of notation, we set the time horizon to be $[0,1]$.

\subsection{Some necessary notations and conditions}
Firstly, we introduce some notations that will be used throughout this section:
\begin{itemize}
	\item For positive integer \( n \in \mathbb{N} \) and \( x, y \in \mathbb{C}^n \), the standard inner product is defined as
	\[
	x \cdot y := \sum_{j=1}^d x_j \bar{y}_j,
	\]
	where \( \bar{y}_j \) is the complex conjugate of \( y_j \). Two common norms on \( \mathbb{C}^n \):
	1-norm: \( |x|_1 := \sum_{j=1}^n |x_j| \);
	sup-norm (\( \ell^\infty \)-norm): \( |x| := \max_{1 \leq j \leq d} |x_j| \).
	\item The $n$-dimensional torus is denoted by $\mathbb{T}^n := \mathbb{R}^n / 2\pi\mathbb{Z}^n$.
	\item For $\alpha > 0$ and $\tau \geq n - 1 \geq 1$, the set of $(\alpha, \tau)$-Diophantine numbers in $\mathbb{R}^n$ is defined as
	\[
	\Delta_{\alpha}^{\tau} := \left\{\omega \in \mathbb{R}^n : |\omega \cdot k| \geq \frac{\alpha}{|k|_1^{\tau}}, ~\forall~ 0 \neq k \in \mathbb{Z}^n \right\}.
	\]
	\item The Lebesgue (outer) measure on $\mathbb{R}^n$ is denoted by $meas$.
	\item For $l \in \mathbb{R}$, the integer part is denoted by $[l]$ and the fractional part by $\{l\}$.
	\item For $l > 0$ and an open subset $A$ of $\mathbb{R}^n$ or $\mathbb{R}^n \times \mathbb{T}^n$, the set $C^l(A)$ consists of continuously differentiable functions $f$ on $A$ up to the order $[l]$ such that $f^{[l]}$ is Hölder-continuous with exponent $\{l\}$ and with finite $C^l$-norm defined by
	\[
	\begin{aligned}
		\|f\|_{C^l(A)} &:= \max\left\{\|f\|_{C^{[l]}(A)}, \|f^{[l]}\|_{C^{\{l\}}(A)}\right\}, \\
		\|f\|_{C^{[l]}(A)} &:= \max_{\substack{k \in \mathbb{N}^d \\ 0 \leq |k|_1 \leq [l]}} \sup_{A} |\partial_I^k f|, \\
		\|f^{[l]}\|_{C^{\{l\}}(A)} &:= \max_{\substack{k \in \mathbb{N}^d \\ |k|_1 = [l]}} \sup_{\substack{I_1, I_2 \in A \\ 0 < |I_1 - I_2| < 1}} \frac{|\partial_I^k f(I_1) - \partial_I^k f(I_2)|}{|I_1 - I_2|^{\{l\}}}.
	\end{aligned}
	\]
	When $A = \mathbb{R}^n$ or $A = \mathbb{R}^n \times \mathbb{T}^n$, we simplify this to $\|f\|_{C^l}$.
	\item For $l > 0$ and any subset $A$ of $\mathbb{R}^n$, $C_W^l(A)$ denotes the set of functions of class $C^l$ on $A$ in the sense of Whitney.
	\item For $r, s > 0$, $y_0 \in \mathbb{C}^n$, and $\varnothing \neq \mathbb{D} \subseteq \mathbb{C}^n$, we define
	\[
	\begin{split}
		\mathbb{T}_s^n &:= \left\{x \in \mathbb{C}^n : |\text{Im }x| < s\right\} / 2\pi\mathbb{Z}^n, \\
		B_r(y_0) &:= \left\{y \in \mathbb{R}^n : |y - y_0| < r\right\}, \quad (y_0 \in \mathbb{R}^n), \\
		D_r(y_0) &:= \left\{y \in \mathbb{C}^n : |y - y_0| < r\right\}, \\
		D_{r, s}(y_0) &:= D_r(y_0) \times \mathbb{T}_s^n, \\
		D_{r, s}(\mathbb{D}) &:= \bigcup_{y_0 \in \mathscr{D}} D_{r, s}(y_0).
	\end{split}
	\]
	\item The unit $(n \times n)$ matrix is denoted by $\mathbb{I}_n := \text{diag}(1)$, and the standard symplectic matrix is given by
	\[
	J := \begin{pmatrix} 0 & -\mathbb{I}_n \\ \mathbb{I}_n & 0 \end{pmatrix}.
	\]
	\item For $\mathbb{D} \subseteq \mathbb{C}^n$, $A_{r, s}(\mathbb{D})$ denotes the Banach space of real-analytic functions with bounded holomorphic extensions to $D_{r, s}(\mathbb{D})$, with norm
	\[
	\|\cdot\|_{r, s, \mathbb{D}} := \sup_{D_{r, s}(\mathbb{D})} |\cdot|.
	\]
	\item The canonical symplectic form on $\mathbb{C}^n \times \mathbb{C}^n$ is given by
	\[
	\omega := \mathrm{d}\theta \land \mathrm{d}I = \mathrm{d}\theta_1 \land \mathrm{d}I_1 + \cdots + \mathrm{d}\theta_n \land \mathrm{d}I_n,
	\]
	and $\phi_H^t$ denotes the associated Hamiltonian flow governed by the Hamiltonian $H(I, \theta)$, $I, \theta \in \mathbb{C}^n$.
	\item The projections on the first and last $d$-components are denoted by $\pi_1 : \mathbb{C}^n \times \mathbb{C}^n \ni (I, \theta) \longmapsto I$ and $\pi_2 : \mathbb{C}^n \times \mathbb{C}^n \ni (I, \theta) \longmapsto \theta$, respectively.
	\item For a linear operator $\mathcal{L}$ from the normed space $(V_1, \|\cdot\|_{V_1})$ into the normed space $(V_2, \|\cdot\|_{V_2})$, its "operator norm" is given by
	\[
	\|\mathcal{L}\| := \sup_{x \in V_1 \setminus \{0\}} \frac{\|\mathcal{L}x\|_{V_2}}{\|x\|_{V_1}},
	\]
	so that $\|\mathcal{L}x\|_{V_2} \leq \|\mathcal{L}\| \|x\|_{V_1}$ for any $x \in V_1$.
	\item For $\omega \in \mathbb{R}^n$ and a $C^1$ function $f$, the directional derivative of $f$ with respect to $\omega$ is given by
	\[
	\mathrm{D}_{\omega} f := \omega \cdot f_I = \sum_{j = 1}^{n} \omega_j f_{I_j}.
	\]
	\item If $f$ is a smooth or analytic function on $\mathbb{T}^n$, its Fourier expansion is given by
	\[
	f = \sum_{k \in \mathbb{Z}^n} f_k e^{i k \cdot \theta}, \quad f_k := \frac{1}{(2\pi)^n} \int_{\mathbb{T}^n} f(\theta) e^{-i k \cdot \theta} \,\mathrm{d}\theta,
	\]
	where $e := \exp(1)$ denotes the Neper number and $i$ the imaginary unit. We also set
	\[
	\langle f\rangle := f_0 = \frac{1}{(2\pi)^n} \int_{\mathbb{T}^n} f(\theta) \,\mathrm{d}\theta.
	\]
\end{itemize}

Moreover, in order to carry out the KAM iteration, we additionally require the following necessary conditions:
\begin{enumerate}
	\item[(H1)] Let \(l > 2\nu := 2(\tau + 1)>2n \geq 4\), and let \(\mathbb{D} \subset \mathbb{R}^n\) be a non-empty, bounded domain.
	\item[(H2)] Consider the Hamiltonian \(H(I, \theta) = H_0(I) + P(I, \theta)\) on the phase space \(\mathbb{D} \times \mathbb{T}^n\). Here, \(H\) and \(P\) are given functions in \(C^l(\mathbb{D} \times \mathbb{T}^n)\) with finite \(l\) - norms \(\|H_0\|_{C^l(\mathbb{D})}\) and \(\eta := \|P\|_{C^l(\mathbb{D} \times \mathbb{T}^n)}\).
	\item[(H3)] Assume that \(H_{0,I}\) is locally uniformly invertible. This implies that for all \(I \in \mathbb{D}\), \(\det H_{0,II}(I) \neq 0\). To simplify notation, define \(T(I) := H_{0,II}(I)^{-1}\), and suppose that \(C_T := \|T\|_{C^0(\mathbb{D})} < \infty\). Furthermore, set \(C_H := \max\left\{1, \|H_{0,II}\|_{C^l(\mathbb{D})}\right\}\), and then define \(\theta := C_T C_H\), which possesses the property that \(\theta \geq 1\).
	\item[(H4)] Let $\alpha \in (0, 1)$ and set,
	\[
	\quad \alpha_{*} := \alpha^{\frac{1}{l-2v}}, \quad
	\mathbb{D}' := \{ I \in \mathbb{D} : B_{\alpha_{*}}(I) \subseteq  \mathbb{D} \}, \quad
	\mathbb{D}_{\alpha} := \{ I \in \mathbb{D}' : H_I(I) \in \Delta^{\tau}_{\alpha} \}.
	\]
	\item[(H5)] Finally, for some suitable constant \(C_1 = C_1(n,l)>1\), set
	\[
	\begin{cases}
		\sigma:=\left(\frac{\eta^{3/2}}{\theta^{2l/\nu}\alpha\sqrt{C_H}}\right)^{1/(l+\nu)},\\
		\rho:=\frac{2C_1 C_H \eta}{\alpha^{2}\sigma^{2\nu}},\\
		\beta_0:=\min\left\{\frac{l}{2\nu}-1+\frac{1}{\nu},2\right\}.
	\end{cases}
	\]
\end{enumerate}

\subsection{Proof of the Theorem \ref{T6.1}}
\begin{proof}
	Under Conditions (C1) and (C2), the Onsager-Machlup functional associated with system $\eqref{1.1}$ follows directly from Theorem \ref{T3.1}. Minimizing this functional shows that the most probable continuous path of the nearly integrable stochastic Hamiltonian system coincides with the solution to the corresponding deterministic nearly integrable Hamiltonian system $\eqref{1.3}$. Since the Hamiltonian admits a nearly integrable structure and the coordinate representations $(q,p)$ and the action--angle variables $(I,\theta)$ are related through a canonical transformation, the argument is essentially the same as in Sections 3 and 4. We therefore omit further details here.
	
	Next, we will outline the framework for proving the existence of invariant tori and the measure of Cantor sets, with detailed proofs referred to in \cite{6}. This commences by extending \(H_0\) and \(P\) to the entirety of the phase space \(\mathbb{R}^n \times \mathbb{T}^n\). To accomplish this extension, we introduce a cut-off function \(\chi \in C(\mathbb{C}^n) \cap C^{\infty}(\mathbb{R}^n)\) that fulfills the conditions \(0 \leq \chi \leq 1\), with its support confined within \(D_{\alpha_*}(\mathbb{D}')\) and \(\chi\) being identically equal to 1 on \(D_{\alpha_*/2}(\mathbb{D}')\). Additionally, for any multi-index \(k \in \mathbb{N}^n\) with \(|k|_1 \leq l\), there exists a constant \(C_0 = C_0(n, l) > 0\) such that
	\[
	\|\partial_y^k\chi\|_{\mathbb{R}^n} \leq C_0\alpha_*^{-|k|_1}.
	\]
	
	Utilizing Faà Di Bruno's Formula \cite{49}, we construct \(\widehat{H}_0 \in C^l(\mathbb{R}^n)\) such that \(\|T\|_{\mathbb{D}}\|\widehat{H}_0 - H_0\|_{C^l(\mathbb{D})} \leq C_1^{-1}\alpha_*^l/4\). Consequently, \(\widehat{H}_{0,II} = H_{0,II}(\mathbb{I}_n + \mathcal{T}(\widehat{H}_{0,II} - H_{0,II}))\) is invertible on \(\mathbb{D}\) with \(\|(\widehat{H}_{0,II})^{-1}\|_{\mathbb{D}} \leq 2\|\mathcal{T}\|_{\mathbb{D}}\).
	
	Defining \(\tilde{H}_0 := \widehat{H}_0 + \chi \cdot (H_0 - \widehat{H}_0)\), we ensure \(\tilde{H}_0 \in C^l(\mathbb{R}^n \times \mathbb{T}^n)\) and \(\tilde{H}_0 \equiv H_0\) on \(D_{\alpha_*/2}(\mathbb{D}')\). Furthermore,
	\[
	\|\tilde{H}_0\|_{C^l} \leq \|H_0\|_{C^l} + C_1\alpha_*^{-l}\|\widehat{H}_0 - H_0\|_{C^l} < 2\|H_0\|_{C^l},
	\]
	and
	\[
	\|(\widehat{H}_{0,II})^{-1}\partial_y^2(\chi \cdot (H_0 - \widehat{H}_0))\|_{\mathbb{D}} \leq \|(\widehat{H}_{0,II})^{-1}\|_{\mathbb{D}} \cdot C_1\alpha_*^{-l}\|\widehat{H}_0 - H_0\|_{C^l} \leq 1/2.
	\]
	Thus, \(\tilde{H}_{0,II}\) is invertible with \(\|(\tilde{H}_{0,II})^{-1}\|_{\mathbb{D}} \leq 2\|(\widehat{H}_{0,II})^{-1}\|_{\mathbb{D}} \leq 4\|T\|_{\mathbb{D}}\).
	
	Similarly, \(P\) is extended to a function \(\tilde{P} \in C^l(\mathbb{R}^n \times \mathbb{T}^n)\) such that \(\tilde{P} \equiv P\) on \(D_{\alpha_*/2}(\mathbb{D}')\) and \(\|\tilde{P}\|_{C^l} \leq 2\|P\|_{C^l}\). Setting \(\tilde{H} := \tilde{H}_0 +  \tilde{P}\), we observe \(\tilde{H}|_{D_{\alpha_*/2}(\mathbb{D}')} = {H}\). Notably, replacing \({H}\) with \(\tilde{H}\) makes no difference since the invariant tori of \(\tilde{H}\) we aim to construct reside within \(D_{\alpha_*/2}(\mathbb{D}')\), given \(r_0 < \alpha_*/2\). 
	
	Let \( \mathbb{D} \subset \mathbb{R}^n \) be a domain with a smooth boundary \( \partial \mathbb{D} \). Suppose there exists a positive constant \( c = c(n, \tau, l) < 1 \) such that the parameters \( \alpha \) and \( \gamma_1 \) satisfy the following conditions
	\[
	0 < \alpha \leq \min\left\{c C_H, \frac{R(\mathbb{D})}{6}, \frac{1}{2} \, \mathrm{minfoc}(\partial \mathbb{D})\right\},\quad \gamma_1 \leq c C_H^{-\frac{l+2v}{l-2v}} \theta^{-a} \alpha^{\frac{2l}{l-2v}},
	\]
	where \( R(\mathbb{D}) := \sup\{R > 0 : B_R(I) \subseteq \mathbb{D} \text{ for some } I \in \mathbb{D}\} \), \( \mathrm{minfoc}(\partial \mathbb{D}) \) is the minimal focal distance of \( \partial \mathbb{D} \) and \( a := (l - 2v)^{-1} \max \left\{(6 + 2lv^{-1})(l + v) - 2l(l - v), \right. \)  \(\left. 2l(l + 3v)v^{-1}\right\} \). 
	
	Let \( \mathcal{H}_j \) (resp.\ \( \mathcal{P}_j \)) denote the real-analytic approximation \( H_{\xi_j} \) (resp.\ \( P_{\xi_j} \)) of \( \tilde{H} \) (resp.\ \( \tilde{P} \)) defined on \( \mathcal{O}_j \), as given by Lemma \ref{L2.15}. Define the initial set \( \mathbb{D}_0 \) as
	\[
	\mathbb{D}_0 := \{I \in \mathbb{R}^n : \partial_I \mathcal{H}_0(I) \in \partial_I H_0(\mathbb{D}_{\alpha})\}.
	\]
	For \( j \geq 1 \), we define the proposition \( (\mathscr{P}_j) \) as follows: there exist
	\begin{enumerate}
		\item A sequence of sets \( \mathbb{D}_j \),
		\item A sequence of diffeomorphisms \( G_j: D_{\tilde{r}_j}(\mathbb{D}_{j - 1}) \to G_j(D_{\tilde{r}_j}(\mathbb{D}_{j - 1})) \),
		\item A sequence of real-analytic symplectic transformations
		\[
		\Phi_j = (v_j, u_j): D_{r_j, s_j}(\mathbb{D}_j) \to D_{\sigma_{j - 1}, \sigma_{j - 1}}(\mathbb{D}_{j - 1}),
		\]
	\end{enumerate}
	such that, setting \( \mathcal{H}_{j - 1} := \mathcal{H}_{0,j - 1} + \mathcal{P}_{j - 1} \), the following properties hold:
	\begin{align*}
		&G_j(\mathbb{D}_{j - 1}) = \mathbb{D}_j \subset \mathbb{D}_{r_j}, \quad G_j = (\partial_I H_{0,j})^{-1} \circ \partial_I H_{0,j - 1},\\
		&\det \partial_I^2 H_{0,j}(y) \neq 0, \quad T_j(I) := \partial_I^2 H_{0,j}(I)^{-1}, \quad \forall I \in \mathbb{D}_j,\\
		&H_j := \mathcal{H}_{j - 1} \circ \phi^j := H_{0,j} + P_j \quad \text{on } D_{r_j, s_j}(\mathbb{D}_j),
	\end{align*}
	where \( \phi^j := \phi_1 \circ \phi_2 \circ \cdots \circ \phi_j \) and \( K_0 := \mathcal{K}_0 \).

	Moreover,
	\begin{align*}
		&\left\lVert G_j - \text{id} \right\rVert_{\tilde{r}_j, \mathbb{D}_{j - 1}} \leq \tilde{r}_j \xi^{2\nu} \xi^{m(j - 1)}, \quad  \left\lVert \partial_y G_j - \mathbb{I}_n \right\rVert_{\tilde{r}_j, \mathbb{D}_{j - 1}} \leq \xi^{2\nu} \xi^{2m(j - 1)},\\
		&\left\lVert \partial_I^2 H_j \right\rVert_{r_j, \mathbb{D}_j} < 2C_H,  \quad \left\lVert T_j \right\rVert_{\mathbb{D}_j} < 2C_T,  \quad  T_j := (\partial_I^2 H_j)^{-1}, \quad
		\left\lVert P_j \right\rVert_{r_j, s_j, \mathbb{D}_j} \leq C_1 C_H \xi_j^{l - 1},\\
		&\max \left\{ \left\lVert M_j (\phi_j - \text{id}) \right\rVert_{2r_j, s_j, \mathbb{D}_j}, \left\lVert \pi_2 \partial_x (\phi_j - \text{id}) \right\rVert_{2r_j, s_j, \mathbb{D}_j} \right\} \leq \xi^{2\nu} \xi^{m(j - 1)},
	\end{align*}
	where \(M_j := \text{diag}(r_j^{-1} \mathbb{I}_n, \sigma_j^{-1} \mathbb{I}_n)\).
	
	We will use mathematical induction to prove that the proposition \((\mathscr{P}_j)\) holds for all \(j \geq 1\). The proof of this part primarily relies on the application of Lemma \ref{L2.14}. It is straightforward to verify that Lemma \ref{L2.14} can be applied to \(\mathcal{H}_0\), which implies that the statement \((\mathscr{P}_1)\) holds.
	
	Next, we assume that \((\mathscr{P}_j)\) holds for some \(j \geq 1\) and proceed to prove \((\mathscr{P}_{j + 1})\). First, observe the following estimates:
	\begin{align*}
		&s_j + \frac{\sigma_{j - 1}}{3} = \frac{(12\xi + 1)\sigma_{j - 1}}{3} < \frac{2\sigma_{j - 1}}{3} < \frac{s_{j - 1}}{2},\\
		&2r_j + \frac{r_{j - 1}\sigma_{j - 1}}{3} < \frac{r_{j - 1}}{4} + \frac{r_{j - 1}}{6} < \frac{r_{j - 1}}{2},\\
		&2r_j + \frac{r_{j - 1}\sigma_{j - 1}}{3} < \sigma_0^{\nu}\xi^j + \frac{\sigma_{j - 1}}{3} = \sigma_0^{\tau}\sigma_j + \frac{\sigma_{j - 1}}{3} < \sigma_{j - 1}.
	\end{align*}
	These inequalities, combined with a symplectic change of coordinates \(\phi' = \text{id} + \tilde{\phi}: D_{\bar{r}/2, s'}(\mathbb{D}') \to D_{\bar{r} + r \sigma / 3, \bar{s}}(\mathbb{D})\) in Lemma \ref{L2.14}, imply that
	\begin{equation}\label{6.4}
		\phi_j(D_{r_j, s_j}(\mathbb{D}_j)) \subset D_{\sigma_{j - 1}, \sigma_{j - 1}}(\mathbb{D}_{j - 1}) \bigcap D_{r_{j - 1}/2, s_{j - 1}/2}(\mathbb{D}_{j - 1}).
	\end{equation}
	In particular, the real-analytic symplectic transformation
	\[
	\phi_j = (v_j, u_j): D_{r_j, s_j}(\mathbb{D}_j) \to D_{\sigma_{j - 1}, \sigma_{j - 1}}(\mathbb{D}_{j - 1}),
	\]
	exists. Furthermore, the inequality
	\[
	2r_{j + 1} < \frac{1}{4} \min \left\{ \frac{\alpha}{2d(2C_H)\kappa_j^{\nu}}, \check{r}_{j + 1} \right\}
	\]
	holds. Together with the definitions of the sequences of the various parameters, this ensures that condition \eqref{2.2} in Lemma \ref{L2.14} is satisfied for all \(j \geq 1\).

	Write
	\[
	\mathcal{H}_j := \mathcal{H}_{0,j} + \mathcal{P}_j = \mathcal{H}_{j - 1} + (\mathcal{H}_{0,j} - \mathcal{H}_{0,j - 1}) + (\mathcal{P}_j - \mathcal{P}_{j - 1})
	\]
	By the inductive assumption and \eqref{6.4}, we have
	\begin{align*}
		\mathcal{H}_j \circ \phi^j &= \mathcal{H}_{j - 1} \circ \phi^j + (\mathcal{H}_{0,j} - \mathcal{H}_{0,j - 1}) \circ v^j + (\mathcal{P}_j - \mathcal{P}_{j - 1}) \circ \phi^j \\
		&= H_{0,j} + P_j + (\mathcal{H}_{0,j} - \mathcal{H}_{0,j - 1}) \circ v_j + (\mathcal{P}_j - \mathcal{P}_{j - 1}) \circ \phi^j \\
		&= \mathcal{H}_0^{j} + \mathcal{P}^j, \quad \text{on } D_{r_j, s_j}(\mathbb{D}_j),
	\end{align*}
	where \(\mathcal{H}_0^j := H_{0,j}\) and \(\mathcal{P}^j := P_j + (\mathcal{H}_{0,j} - \mathcal{H}_{0,j - 1}) \circ \phi^j + (\mathcal{P}_j - \mathcal{P}_{j - 1}) \circ \phi^j\), with
	\begin{equation}\label{6.5}
		\left\lVert \partial_y^2 \mathcal{H}_0^j \right\rVert_{r_j, \mathbb{D}_j} < 2C_K, \qquad  \left\lVert (\partial_y^2 \mathcal{H}_0^j)^{-1} \right\rVert_{\mathbb{D}_j} < 2C_T,
	\end{equation}
	by the inductive assumption, provided \(\phi^j\) maps \(D_{r_j, s_j}(\mathbb{D}_j)\) into \(\mathcal{O}_j = \{ (I, \theta) \in \mathbb{C}^n \times \mathbb{C}^n : |\text{Im}(I, \theta)| < \xi_j \}\) i.e.
	\begin{equation}\label{6.6}
		\sup_{D_{r_j, s_j}(\mathscr{D}_j)} |\text{Im} \phi^j| \leq \frac{\xi_j}{2}.
	\end{equation}
	Hence,
	\begin{align*}
		\left\lVert \mathcal{P}^j \right\rVert_{r_j, s_j, \mathbb{D}_j} &\leq \left\lVert P_j \right\rVert_{r_j, s_j, \mathbb{D}_j} + \left\lVert (\mathcal{H}_{0,j} - \mathcal{H}_{0,j - 1}) \circ \phi^j \right\rVert_{r_j, s_j, \mathbb{D}_j} + \left\lVert (\mathcal{P}_j - \mathcal{P}_{j - 1}) \circ \phi^j \right\rVert_{r_j, s_j, \mathbb{D}_j} \\
		&\leq C_1 C_H \xi_{j - 1}^l + \left\lVert \mathcal{H}_{0,j} - \mathcal{H}_{0,j - 1} \right\rVert_{\xi_j} + \left\lVert \mathcal{P}_{0,j} - \mathcal{P}_{o,j - 1} \right\rVert_{\xi_j} \\
		&\leq C_1 C_H \xi_{j - 1}^l + C_1 C_H \xi_{j - 1}^l + C_1 \eta \xi_{j - 1}^l \\
		&\leq 3C_1 C_H \xi_{j - 1}^l.
	\end{align*}
	Thus, thanks to \eqref{6.5}, \(\mathcal{H}_j \circ \phi^j = \mathcal{H}_0^j + \mathcal{P}^j\) satisfies the assumptions in \eqref{2.1} with \(\eta \sim \left\lVert \mathcal{P}^j \right\rVert_{r_j, s_j, \mathbb{D}_j}\), \(r \sim r_j\), \(s \sim s_j\), \(\sigma \sim \sigma_j\), \(C_H \sim 2C_H\) as
	\[
	\partial_I^2 \mathcal{H}_0^j(\mathbb{D}_j) \stackrel{\text{def}}{=} \partial_I^2 H_{0,j}(G_j(\mathbb{D}_{j - 1})) = \partial_I^2 H_{0,j - 1}(\mathbb{D}_{j - 1}) = \cdots = \partial_I^2 H_{0,0}(\mathbb{D}_0) \subset \Delta_{\tau}^{\alpha}.
	\]
	Hence, in order to apply Lemma \ref{L2.14} to \(\mathcal{H}_j \circ \phi^j = \mathcal{H}_0^j + \mathcal{P}^j\), we only need to check \eqref{2.3}. Upon observation, we have
	\begin{align*}
		&r_j = r_0 \xi^{\nu j} \leq \frac{\alpha}{2C_H} \sigma_j^{\nu} \leq \alpha \sigma_j^{\nu} / \left\lVert \partial_I^2 \mathcal{H}_0^j \right\rVert_{r_j, s_j, \mathbb{D}_j},\\
		&\sigma_1^{-\nu} \frac{\left\lVert \mathcal{P}^1 \right\rVert_{r_1, s_1, \mathbb{D}_1}}{\alpha r_1} \rho^{-1} \leq C_2 \sigma_0^l \frac{C_H}{\eta \xi^{2\nu}} \leq 1 ,\\
		&3C_0 \frac{\theta C_T \left\lVert \mathcal{P}^1 \right\rVert_{r_1, s_1, \mathbb{D}_1}}{r_1 \check{r}_2 \bar{\sigma}_1} \leq C_2 \sigma_0^{l - 2\nu} \frac{\theta^{6 + m} (C_H)^2}{\alpha^2} \lambda^{2(\nu + m)} \leq \xi^{2\nu},
	\end{align*}
	and, for \(j \geq 2\),
	\begin{align*}
		&\sigma_j^{-\nu} \frac{\left\lVert \mathcal{P}^j \right\rVert_{r_j, s_j, \mathbb{D}_j}}{\alpha r_j} \rho^{-1} \leq C_2 \sigma_0^l \frac{C_H}{\eta} \xi^{(l - 2\nu)j - 2l} \leq C_2 \sigma_0^l \frac{C_H}{\eta} \xi^{-2l} \leq 1,\\
		&3C_0 \frac{\theta C_T \left\lVert \mathcal{P}^j \right\rVert_{r_j, s_j, \mathbb{D}_j}}{r_j \check{r}_{j + 1} \bar{\sigma}_j} \leq C_2 \sigma_0^{l - 2\nu} \frac{\theta^{4 + 2l / \nu} (C_H)^2}{\alpha^2} \lambda^{2l} \leq \xi^{2\nu}.
	\end{align*}
	Therefore, Lemma \ref{L2.14} applies to \(\mathcal{H}_j\) and yields the desired symplectic change of coordinates \(\phi_{j + 1}\).

	Furthermore, based on Lemmas \ref{L2.15} and \ref{L2.16}, we can obtain the convergence results for \( G^j \), \( P_j \), \( \phi^j \), and \( H_j \) as follows:
	\begin{itemize}
		\item The sequence \( G^j := G_j \circ G_{j - 1} \circ \cdots \circ G_2 \circ G_1 \) converges uniformly on \( \mathbb{D}_0 \) to a diffeomorphism \( G_*: \mathbb{D}_0 \to \mathbb{D}_* := G_*(\mathbb{D}_0) \subset \mathbb{D} \), and \( G_* \in C_W^1(\mathbb{D}_0) \).
		
		\item \( P_j \) converges uniformly to 0 on \( \mathbb{D}_* \times \mathbb{T}_{s_*}^n \) in the \( C_W^2 \) topology.
		
		\item \( \phi^j \) converges uniformly on \( \mathbb{D}_* \times \mathbb{T}^d \) to a symplectic transformation
		\[
		\phi_*: \mathbb{D}_* \times \mathbb{T}^n \stackrel{\text{into}}{\longrightarrow} \mathbb{D} \times \mathbb{T}^d,
		\]
		with \( \phi_* \in C_W^{\tilde{m}}(\mathbb{D}_* \times \mathbb{T}^n) \) and \( \phi_*(\cdot, \cdot) \in C^{\tilde{m}v}(\mathbb{T}^n) \), for any given \( y_* \in \mathbb{D}_* \).
		
		\item \( H_{0,j} \) converges uniformly on \( \mathbb{D}_* \) to a function \( H_* \in C_W^{2 + \tilde{m}}(\mathbb{D}_*) \), with
		\begin{align*}
			&\partial_{I_*} H_* \circ G_* = \partial_I \mathcal{H}_0, \qquad \qquad\qquad \text{on } \mathbb{D}_0, \\
			&H \circ \phi_*(I_*, x) = H_*(I_*), \qquad \forall (I_*, x) \in \mathbb{D}_* \times \mathbb{T}^n.
		\end{align*}
	\end{itemize}

	Based on the preceding proof, we can demonstrate that there exists a Cantor-like set \( \mathbb{D}_* \subset \mathbb{D} \), an embedding \( \phi_* = (v_*, u_*): \mathbb{D}_* \times \mathbb{T}^n \rightarrow \mathcal{H} := \phi_*(\mathbb{D}_* \times \mathbb{T}^n) \subset \mathbb{D} \times \mathbb{T}^n \) of class \( C_W^{\beta}(\mathbb{D}_* \times \mathbb{T}^n) \), and a function \( H_* \in C_W^2(\mathbb{D}_*, \mathbb{R}) \), such that \( H \circ \phi_*(I_*, \theta) = K_*(I_*, \theta) \) for all \( (I_*, \theta) \in \mathbb{D}_* \times \mathbb{T}^n \). The map \( \xi \mapsto \phi_*(I_*, \theta) \) is of class \( C^{\beta v}(\mathbb{T}^n) \) for any \( I_* \in \mathbb{D}_* \) (with \( v^{-1} < \beta < \beta_0 \)), and the map \( G^* := (\partial_{I_*} H_*)^{-1} \circ \partial_I H: \mathbb{D}_{\alpha} \rightarrow \mathbb{D}_* \) defines a lipeomorphism onto \( \mathbb{D}_* \), satisfying \( B_{\alpha_*/2}(\mathbb{D}_*) \subseteq \mathbb{D} \). The set \( \mathcal{K} \) is foliated by KAM tori of \( H \), each being a graph of a \( C^v(\mathbb{T}^n) \)-map.  
	
	Furthermore, the following estimates hold
	\[
	\|G^* - \mathrm{id}\|_{\mathcal{D}_*} \leq \eta^{\frac{3\tau}{2(l+v)}} \alpha^{\frac{l+1}{l+v}} C_H^{\frac{\tau}{2(l+v)}} \theta^{-1} L^{\frac{2l\tau}{v(l+v)}}, \quad \|G^* - \mathrm{id}\|_{L, \mathbb{D}_{\alpha}} < \frac{1}{2},
	\]  
	and  
	\[
	\sup_{\mathbb{D}_* \times \mathbb{T}^d} \max\left\{ |M(\phi_* - \mathrm{id})|, \, \|\pi_2(\partial_x \phi_* - \mathbb{I}_n)\| \right\} \leq 8\theta^{-2} (\log \rho^{-1})^{-2v} < 1,
	\]  
	where \( M := \mathrm{diag}\left(C_H(\alpha \sigma^v)^{-1} \mathbb{I}_n, \sigma^{-1} \mathbb{I}_n\right) \).  
	
	Then, the measure of the complement of \( \mathcal{K} \) is bounded by  
	\[
	\mathrm{meas}(\mathbb{D} \times \mathbb{T}^n \setminus \mathcal{K}) \leq (3\pi)^n \left(2\mathcal{H}^{n-1}(\partial \mathbb{D}) \tilde{\eta} + C \tilde{\eta}^2 + \mathrm{meas}(\mathbb{D}' \setminus \mathbb{D}_\alpha)\right),
	\]  
	with \[ \tilde{\eta} := \max\left\{ \eta^{\frac{3\tau}{2(l+v)}} \alpha^{\frac{l+1}{l+v}} C_H^{\frac{-\tau}{2(l+v)}} \theta^{-1} L^{\frac{2l\tau}{v(l+v)}}, \alpha_* \right\},
	C = 2 \sum_{j=1}^{\left\lfloor \frac{n-1}{2} \right\rfloor} \frac{\tilde{\eta}^{2j-1} {k}_{2j}(\mathbb{R}^{\partial \mathbb{D}})}{1 \cdot 3 \cdots (2j+1)},
	\]  
	where \( {k}_{2j}(\mathbb{R}^{\partial \mathbb{D}}) \) denotes the \( (2j) \)-th integrated mean curvature of \( \partial \mathbb{D} \).

	In summary, we have completed the proof of the persistence of KAM tori for stochastic Hamiltonian systems, in the sense of the most probable path. Finally, we investigate the limiting behavior of the stochastic Hamiltonian system as the noise intensity tends to zero. By Theorem~\ref{T5.1}, the nearly integrable stochastic Hamiltonian system satisfies a large deviation principle: as $\gamma \to 0$, the most probable path $\hat{\varphi}(t)=\big(\hat{\varphi}_\theta(t),\hat{\varphi}_I(t)\big)$ is given by the solution to the deterministic nearly integrable Hamiltonian system \eqref{1.3}. Moreover, the family of probability measures on the path space induced by $X^\gamma(\cdot)$ satisfies a large deviation principle with speed $\gamma^2$ and rate function $J(\cdot)$ defined by
	\begin{displaymath}
		\begin{aligned}
			J(\varphi)
			&=\frac{1}{2}\left(
			\int_0^1 \left| \sigma_{\theta}^{-1}(t)\left( \dot{\varphi}_{\theta}(t)
			-\omega(\varphi_{I}(t))
			-\frac{\partial P}{\partial I}(\varphi_{I}(t),\varphi_{\theta}(t)) \right)\right|^2 \,{\rm d}t \right.\\
			&\qquad\left.
			+ \int_0^1 \left| \sigma_{I}^{-1}(t)\left( \dot{\varphi}_{I}(t)
			+ \frac{\partial P}{\partial \theta}(\varphi_{I}(t),\varphi_{\theta}(t)) \right)\right|^2 \,{\rm d}t
			\right).
		\end{aligned}
	\end{displaymath}
	Here $\gamma$ denotes the intensity of the stochastic perturbation, and $\sigma_\theta^{-1}(t)$ and $\sigma_I^{-1}(t)$ are the inverses of the diffusion matrices $\sigma_\theta(t)$ and $\sigma_I(t)$, respectively.
\end{proof}

This result aligns with the conclusions of the deterministic KAM theory, further demonstrating that most invariant tori can survive under small perturbations. In a stochastic setting, these tori are preserved in a probabilistic sense, providing new theoretical insights into the stability of nearly integrable Hamiltonian systems under stochastic perturbations. Based on Theorem 6.1, we derive an intriguing corollary.

\begin{corollary}\label{C6.2}
	Under the assumptions of Theorem \ref{T6.1}, let $\mathcal{K}_0$ denote the family of invariant tori of the integrable Hamiltonian system \eqref{1.2}, let $\mathcal{K}$ denote the family of invariant tori of the nearly integrable Hamiltonian system \eqref{1.3}, and let $X^{\gamma}(t)$ be the solution to the randomly almost integrable Hamiltonian system \eqref{1.1}. By Theorem \ref{T6.1}, if the frequency vector satisfies the $(\alpha,\tau)$-Diophantine condition, then the ``most probable path'' of the stochastic system $X^{\gamma}(t)$ concentrates on the invariant tori $\mathcal{K}$; that is, in the most probable sense, the stochastic perturbation does not destroy the torus structure. Moreover, the probability that $X^{\gamma}(t)$ remains on the original family of invariant tori $\mathcal{K}_0$ satisfies
	\begin{displaymath}
		\mathbb{P}(X^{\gamma}(t) \in \mathcal{K}_0) \approx \exp \left\lbrace  -C\frac{\eta^2}{{\gamma}^2} \right\rbrace
		\times \mathbb{P}\left( \left\| W^{\sigma} \right\| \leq 1 \right),
	\end{displaymath}
	where $\|\cdot\|$ denotes the sup norm. This estimate clearly reveals the ``slight deformation'' phenomenon described in Theorem \ref{T6.1}. It not only quantifies the relative likelihood that the solution $X^{\gamma}(t)$ lies in the vicinity of $\mathcal{K}_0$ versus $\mathcal{K}$, but also shows from a probabilistic viewpoint that $\mathcal{K}$ constitutes the ``most probable'' continuation and geometric deformation of $\mathcal{K}_0$ under the combined effects of deterministic and stochastic perturbations.
\end{corollary}
\begin{proof}
	Based on the content of Theorem $\ref{T6.1}$, we have
	\begin{align*}
		&\quad \frac{\mathbb{P}(X^{\gamma}(t) \in \mathcal{K}_0) }{\mathbb{P}\left(  \left\| W^{\sigma} \right\| \leq 1\right) } \approx \exp\left\lbrace -\frac{1}{{\gamma}^2} \inf_{\varphi \in \mathcal{K}_0} J(\varphi)\right\rbrace\\
		& = \exp\left\lbrace  -\frac{1}{2 {\gamma}^2} \inf_{\varphi \in \mathcal{K}_0} \left( \int_0^1 \left| \sigma_\theta^{-1}(t) \left( \frac{\partial P}{\partial \varphi_I}(\varphi_\theta, \varphi_I)\right) \right|^2 \,{\rm d}t + \int_0^1 \left| \sigma_I^{-1}(t) \left( \frac{\partial P}{\partial \varphi_\theta}(\varphi_\theta, \varphi_I) \right) \right|^2 \,{\rm d}t \right) \right\rbrace\\
		& = \exp\left\lbrace -C\frac{\eta^2}{{\gamma}^2} \right\rbrace,
	\end{align*}
	where $C$ is a quantity that depends on $\sigma_\theta^{-1}$, $\sigma_I^{-1}$ and the partial derivatives of $H$.
\end{proof}

\subsection{A example}
\begin{example}
	To further illustrate the above theory, we introduce a two-dimensional stochastic oscillator equation as a concrete example. This system describes two coupled harmonic oscillators under both deterministic and stochastic perturbations. The stochastic oscillator equations are given as follows:
	\begin{equation}\label{30}
		\begin{aligned}
			\begin{cases}
				\mathrm{d}q_1(t) = p_1(t)\,\mathrm{d}t - \gamma p_2(t) \,\mathrm{d}t + \gamma (2 + \sin(t))\,\mathrm{d}W_1(t), \\
				\mathrm{d}p_1(t) = - 2 q_1(t)\,\mathrm{d}t + \gamma \sin(0.6 t)\,\mathrm{d}t + \gamma (2 + \cos(t))\,\mathrm{d}W_2(t), \\
				\mathrm{d}q_2(t) = p_2(t)\,\mathrm{d}t - \gamma p_1(t) \,\mathrm{d}t + \gamma (1 + 2\sin(t))\,\mathrm{d}W_3(t), \\
				\mathrm{d}p_2(t) = - q_2(t)\,\mathrm{d}t + \gamma \cos(0.6 t)\,\mathrm{d}t + \gamma (1 + 2\cos(t))\,\mathrm{d}W_4(t).
			\end{cases}
		\end{aligned}
	\end{equation}
	In this system, \( \gamma \) represents the perturbation coefficient, and \( W_1(t), W_2(t), W_3(t), W_4(t) \) are independent Wiener processes that introduce stochastic perturbations into the system. Here, \( q_1(t) \) and \( q_2(t) \) are the generalized coordinates, and \( p_1(t) \) and \( p_2(t) \) are the corresponding momenta. The presence of the stochastic terms makes the system non-deterministic, subject to random noise.
\end{example}

When the stochastic terms disappear, the system reduces to the following deterministic Hamiltonian system:
\begin{equation}\label{31}
	\begin{aligned}
		\begin{cases}
			\mathrm{d}q_1(t) = p_1(t)\,\mathrm{d}t - \gamma p_2(t) \,\mathrm{d}t, \\
			\mathrm{d}p_1(t) = - 2 q_1(t)\,\mathrm{d}t + \gamma \sin(0.6 t)\,\mathrm{d}t, \\
			\mathrm{d}q_2(t) = p_2(t)\,\mathrm{d}t - \gamma p_1(t) \,\mathrm{d}t, \\
			\mathrm{d}p_2(t) = - q_2(t)\,\mathrm{d}t + \gamma \cos(0.6 t)\,\mathrm{d}t.
		\end{cases}
	\end{aligned}
\end{equation}
This is a typical nearly integrable Hamiltonian system, where \( \gamma \) represents a small deterministic perturbation. Without stochastic perturbations, the system exhibits classic harmonic oscillatory behavior, with the relationship between the generalized coordinates and momenta governed by the Hamiltonian.

The Hamiltonian of this system can be written as:
\begin{displaymath}
	H(q_1, q_2, p_1, p_2, t) = \frac{p_1^2}{2} + \frac{p_2^2}{2} + q_1^2 + \frac{q_2^2}{2} - \gamma \left( q_1 \sin(0.6t) + q_2 \cos(0.6t) + p_1 p_2 \right),
\end{displaymath}
where \( q_1, q_2 \) are the generalized coordinates and \( p_1, p_2 \) are their corresponding momenta. The term \( \gamma q_1 p_2 \sin(t) \) represents the perturbation, which introduces coupling between the two oscillators and depends on time \( t \). This coupling alters the energy distribution within the system, leading to mutual influence between the two oscillators.

From the theorems discussed earlier, we know that the most probable continuous path of the stochastic nearly integrable Hamiltonian system \eqref{30} is governed by the deterministic equations in system \eqref{31}. As \( \gamma \to 0 \), the path of system \eqref{30} satisfies the large deviation principle, and we can quantify the probability distribution of the system's deviation from the most probable path using the rate function.

To gain a better understanding of the system's behavior under different perturbation strengths and to validate our theoretical results, we performed numerical simulations. Specifically, we considered three different perturbation strengths: \( \gamma = 0.001 \), \( \gamma = 0.01 \), and \( \gamma = 0.1 \). The results of these simulations are shown in the figures below.

\begin{figure}[htbp]
	\fontsize{10}{12}
	\centering
	\includegraphics[trim={1.8cm 2cm 1.5cm 3cm}, clip, width=12cm]{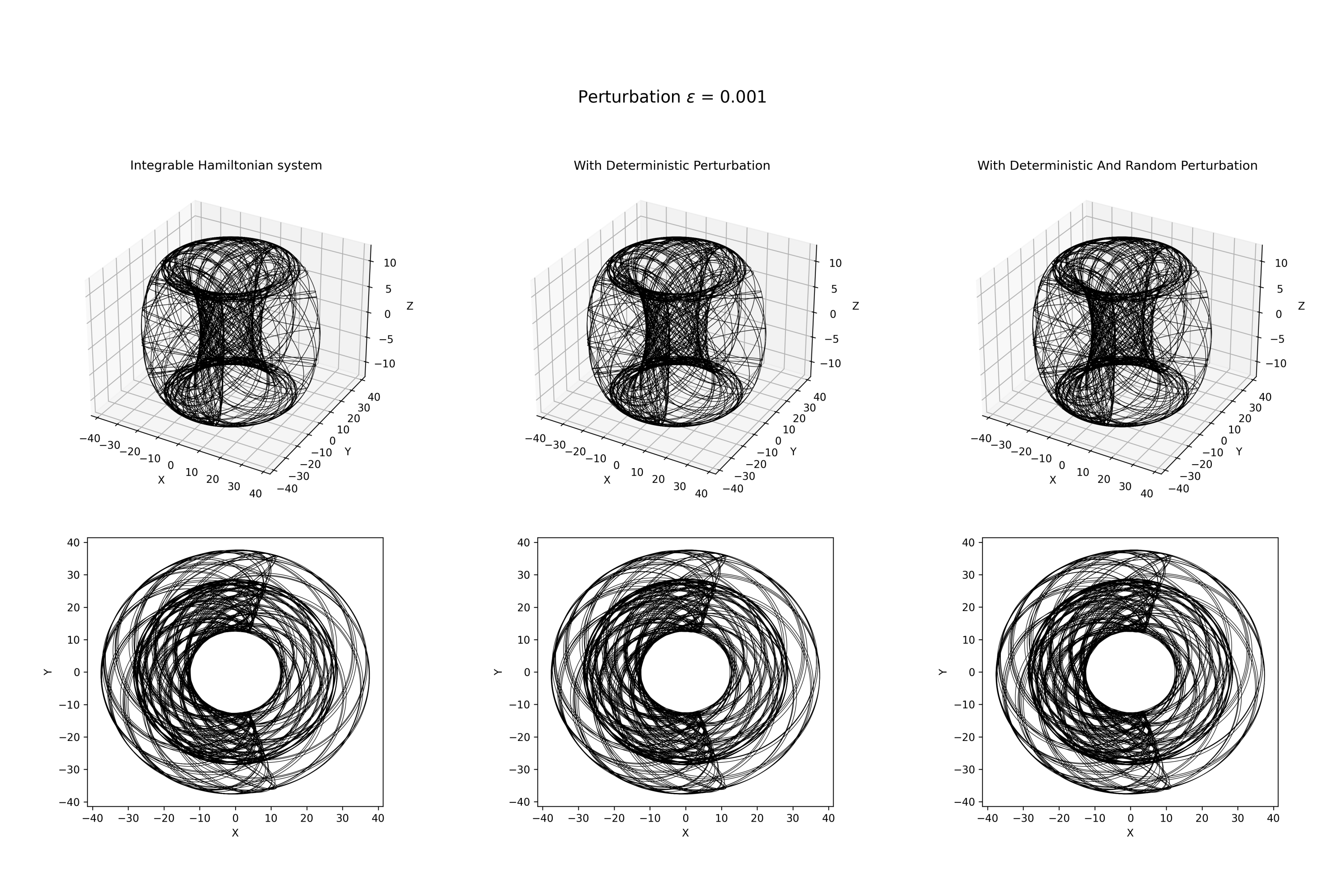}
	\caption{\selectfont
		The first row of three plots represents the phase space trajectories of the solutions to the stochastic nearly integrable Hamiltonian system $\eqref{30}$, the nearly integrable Hamiltonian system $\eqref{31}$, and the corresponding integrable Hamiltonian system, respectively, under a perturbation strength $ \gamma = 0.001 $. The second row of plots shows the corresponding projections of the trajectories from the first row onto the X-Y plane.}
	\label{F.3}
\end{figure}

At very small perturbations, the system closely resembles an integrable Hamiltonian system, see Fig $\ref{F.3}$. The invariant tori are well-preserved, and the trajectories in phase space exhibit regular, closed curves. Even with the introduction of deterministic and stochastic perturbations, the system's trajectory remains largely stable, with minimal impact from the stochastic terms.

\begin{figure}[htbp]
	\fontsize{10}{12}
	\centering
	\includegraphics[trim={1.8cm 0.8cm 1.5cm 3cm}, clip, width=12cm]{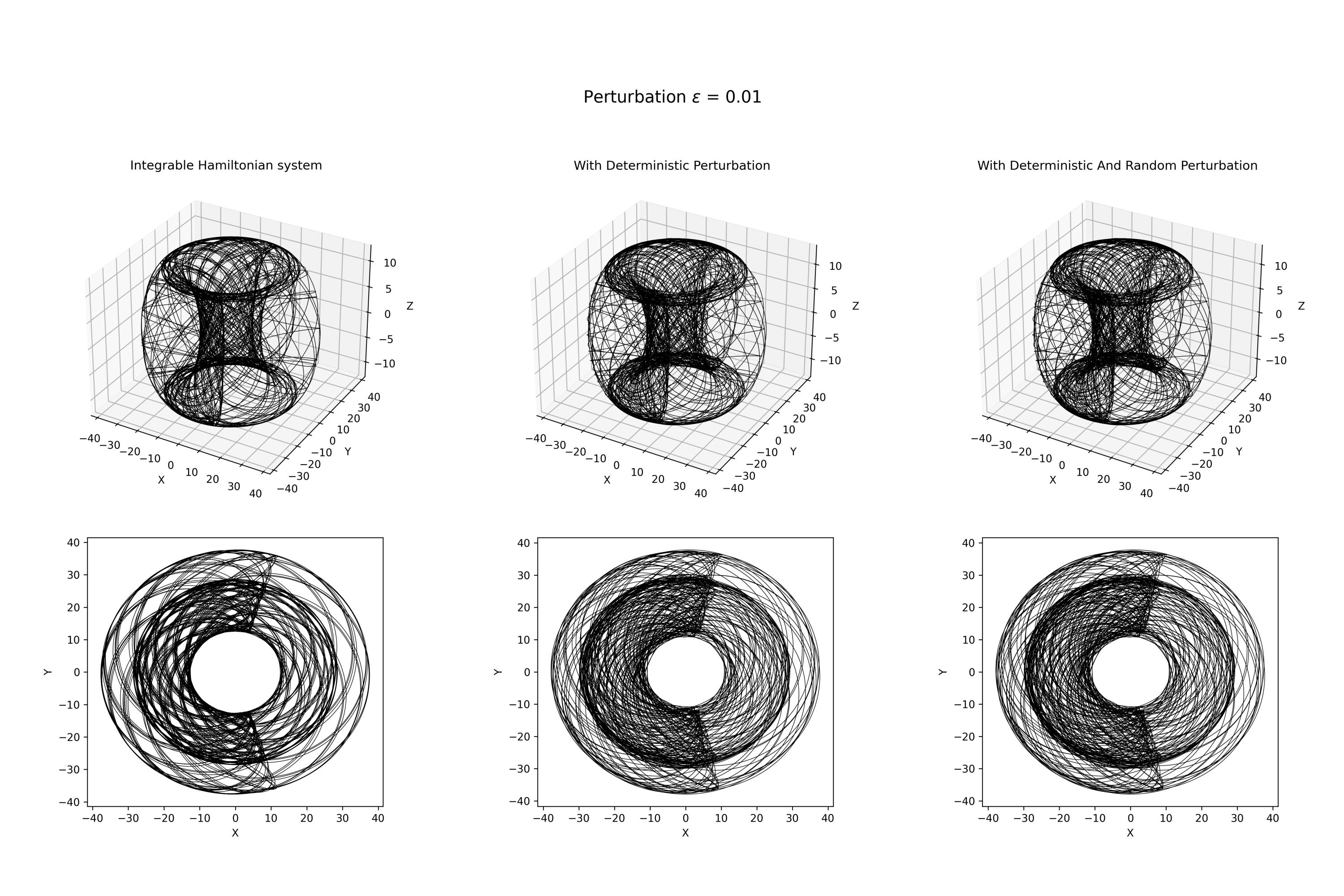}
	\includegraphics[trim={1.8cm 2cm 1.5cm 3cm}, clip, width=12cm]{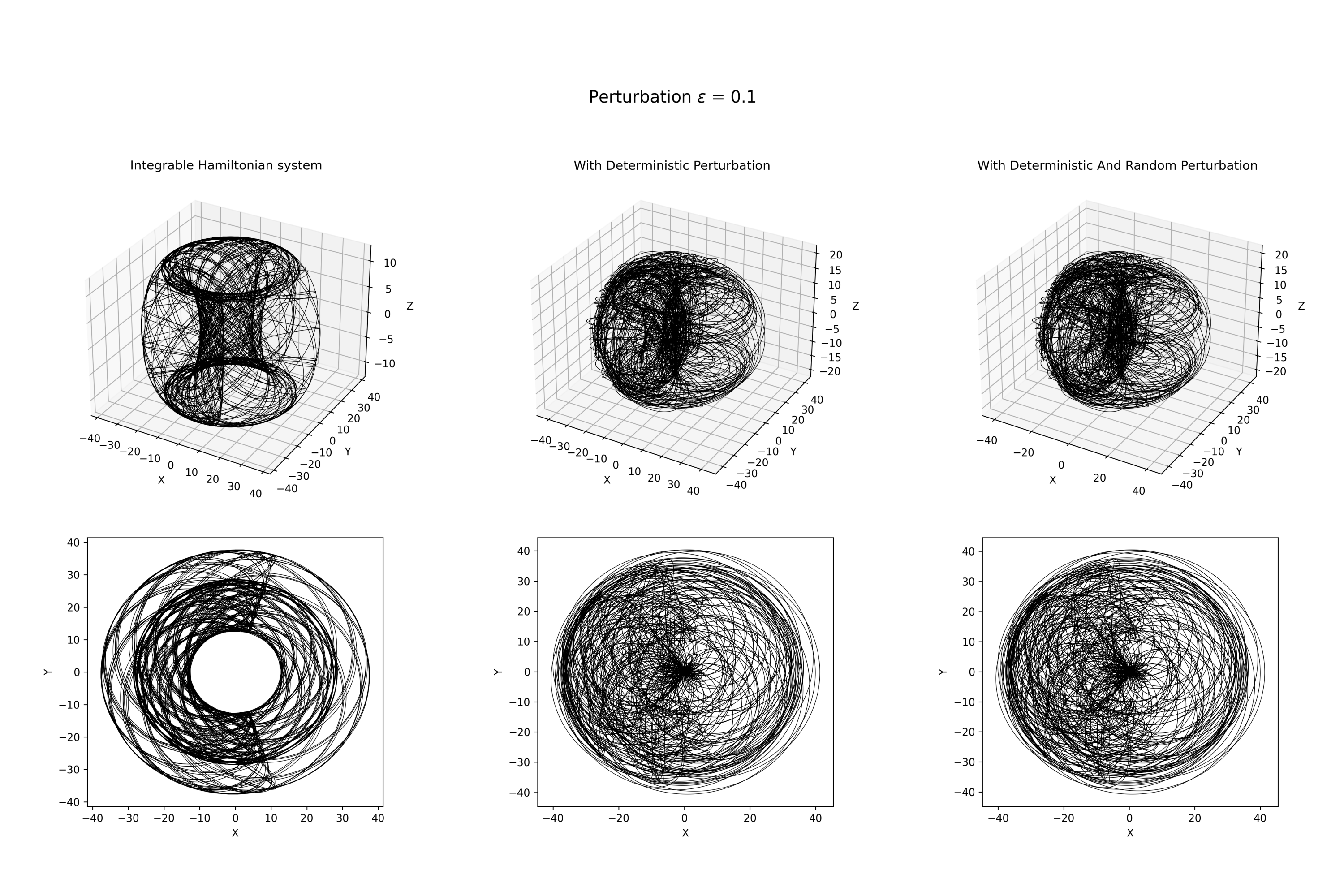}
	\caption{\selectfont The comparison figure of Fig $\ref{F.3}$ when $ \gamma = 0.01 $ and $ \gamma = 0.1 $, respectively.}
	\label{F.4}
\end{figure}

As the perturbation strength increases, the system's trajectories begin to change, as shown in Fig $\ref{F.4}$. While the invariant tori are still present in phase space, the combined effects of deterministic and stochastic perturbations lead to increased complexity in the trajectories. The stability of these trajectories gradually decreases, though they still exhibit quasi-periodic behavior. When the perturbation strength is further increased, the complexity grows significantly, with stochastic perturbations introducing more fluctuations that result in chaotic trajectories. This indicates that when the perturbations are strong enough, the overall topological structure of the system becomes disrupted, ultimately leading to the destruction of the invariant tori.

From the results of these numerical simulations, we can observe that under small stochastic perturbations, the system's trajectories generally evolve along deterministic paths, and the invariant tori are well-preserved. However, as the perturbation strength increases, the system's trajectories become progressively more complex. The fluctuations introduced by stochastic perturbations become more pronounced, especially when \( \gamma = 0.1 \), where the system exhibits stronger chaotic behavior.

Combining theoretical analysis with numerical simulation results, we can draw the conclusion that when the perturbation coefficient is small, the preservation of the invariant torus in the almost integrable stochastic Hamiltonian system can be guaranteed in a probabilistic sense. At this time, the trajectory of the system evolves along the solution of the deterministic Hamiltonian equation, and the probability of deviating from the most likely path decays exponentially. Although stochastic perturbations introduce additional complexity, the essential structure of the system can still be preserved in the sense of the most probable path. Moreover, provided that the perturbation is sufficiently small, the probability that the stochastic system preserves invariant tori is correspondingly high.

\newpage
\section*{acknowledgments} 
The reviewers' comments were extremely helpful. They identified and corrected a key issue in our manuscript, thereby improving its correctness and rigor. We are also grateful for the editor’s support, which gave us the opportunity to revise and strengthen the paper. We would like to express our sincere thanks.

\bibliographystyle{amsplain}
\bibliography{Ref}

\end{document}